\documentclass[11pt, leqno, twoside]{article}

\usepackage{amssymb}
\usepackage{amsmath}
\usepackage{amsthm}
\usepackage{amsfonts}
\usepackage{mathrsfs}
\usepackage{indentfirst}

\usepackage{color}
\def\red{\color{red}}

\allowdisplaybreaks

\usepackage{txfonts}

\def\vpz{\vphantom}

\def\rr{{\mathbb R}}
\def\rn{{{\rr}^n}}
\def\rnn{{\rr}^{n+1}_+}
\def\zz{{\mathbb Z}}
\def\cc{{\mathbb C}}
\def\nn{{\mathbb N}}
\def\cp{{\mathcal P}}

\def\cf{{\mathcal F}}

\def\cm{{\mathcal M}}

\def\ca{{\mathcal A}}

\def\fz{\infty}
\def\az{\alpha}
\def\bz{\beta}

\def\gz{{\gamma}}
\def\bgz{{\Gamma}}

\def\lz{\lambda}

\def\oz{{\omega}}

\def\tz{\theta}

\def\vz{\varphi}
\def\uc{{\varepsilon}}

\def\lf{\left}
\def\r{\right}

\def\hs{\hspace{0.26cm}}

\def\ls{\lesssim}

\def\noz{\nonumber}
\def\wz{\widetilde}
\def\wh{\widehat}
\def\st{\subset}
\def\com{\complement}
\def\bh{\backslash}
\def\btd{\bigtriangledown}

\def\cs{{\mathcal S}}

\def\gfz{\genfrac{}{}{0pt}{}}

\def\dist{\mathop\mathrm{\,dist\,}}
\def\supp{\mathop\mathrm{\,supp\,}}

\def\loc{{\mathop\mathrm{\,loc\,}}}

\def\q1{\wz q}
\def\Q1{q_1}
\def\lv{{L^{p(\cdot)}(\rn)}}
\def\wlv{W\!L^{p(\cdot)}(\rn)}

\def\whv{{W\!H^{p(\cdot)}(\rn)}}
\def\wha{{W\!H_{L,\mathrm{at},M}^{p(\cdot)}(\rn)}}
\def\whf{{W\!H_{L,\mathrm{max}}^{p(\cdot),\mathcal F}(\rn)}}

\def\Bij{{B_{i,j}}}
\def\wBij{\wz B_{i,j}}
\def\bij{{b_{i,j}}}

\def\lij{{\lz_{i,j}}}
\def\aij{{a_{i,j}}}
\def\mij{{m_{i,j}}}

\def\vp{{L^{p(\cdot)}(\rn)}}
\def\ujb{{U_j(B)}}
\def\mol{\mathbb{W\!H}_{L,\,M}^{p(\cdot),\,\uc}(\mathbb{R}^n)}
\def\mhp{W\!H_{L,\,M}^{p(\cdot),\,\uc}(\mathbb{R}^n)}

\def\supp{{\mathop\mathrm{\,supp\,}}}

\def\dist{{\mathop\mathrm{\,dist\,}}}
\def\loc{{\mathop\mathrm{loc\,}}}

\newtheorem{thm}{Theorem}[section]
\newtheorem{prop}[thm]{Proposition}
\newtheorem{lem}[thm]{Lemma}
\newtheorem{cor}[thm]{Corollary}
\theoremstyle{definition}
\newtheorem{defn}[thm]{Definition}
\newtheorem{rem}[thm]{Remark}
\newtheorem{assumption}[thm]{Assumption}

\numberwithin{equation}{section}
\numberwithin{equation}{section}

\begin{document}

\title{\bf\Large Variable Weak Hardy Spaces $W\!H_L^{p(\cdot)}(\rn)$ Associated with
 Operators Satisfying Davies-Gaffney Estimates
\footnotetext{\hspace{-0.35cm} 2010 {\it
Mathematics Subject Classification}. Primary 42B30;
Secondary 42B35, 42B25.
\endgraf {\it Key words and phrases.} weak Hardy space, variable exponent, operator,
Davies-Gaffney estimate, atom, molecule, maximal function.
\endgraf The first author is supported by the Construct Program of the Key Discipline in
Hunan Province, the Scientific Research Fund of
Hunan Provincial Education Department (Grant No. 17B159)
and Hunan Natural Science Foundation
(Grant No: 2018JJ3321). This project is also supported by the National
Natural Science Foundation of China
(Grant Nos. 11701174, 11571039, 11761131002 and 11726621).}}
\author{Ciqiang Zhuo and Dachun Yang\,\footnote{Corresponding author/{\red May 20, 2018}/Final Version}}
\date{ }
\maketitle

\vspace{-0.8cm}

\begin{center}
\begin{minipage}{13cm}
{\small {\bf Abstract}\quad
Let $p(\cdot):\ \mathbb R^n\to(0,1]$ be a variable exponent function
satisfying the globally log-H\"older continuous condition and $L$ a one to one operator of
type $\omega$ in $L^2(\rn)$, with $\oz\in[0,\,\pi/2)$,
which has a bounded holomorphic functional calculus and satisfies the Davies-Gaffney estimates.
In this article, the authors introduce the variable weak Hardy space
$W\!H_L^{p(\cdot)}(\mathbb R^n)$ associated with $L$ via the corresponding square function.
Its molecular characterization
is then established by means of the atomic decomposition of the variable weak tent space
$W\!T^{p(\cdot)}(\mathbb R^n)$ which is also obtained in this article.
In particular, when $L$ is non-negative and self-adjoint, the authors obtain the
atomic characterization of $W\!H_L^{p(\cdot)}(\mathbb R^n)$.
As an application of the molecular characterization,
when $L$ is the second-order divergence form elliptic operator
with complex bounded measurable coefficient, the authors prove that the associated
Riesz transform $\nabla L^{-1/2}$ is bounded from $W\!H_L^{p(\cdot)}(\mathbb R^n)$
to the variable weak Hardy space $W\!H^{p(\cdot)}(\mathbb R^n)$.
Moreover, when $L$ is non-negative and self-adjoint with the kernels
of $\{e^{-tL}\}_{t>0}$ satisfying the Gauss upper bound estimates,
the atomic characterization of $W\!H_L^{p(\cdot)}(\mathbb R^n)$
 is further used to characterize the space via non-tangential maximal functions.
}
\end{minipage}
\end{center}

\section{Introduction\label{s-intro}}

The variable Lebesgue space $\lv$
is a generalization of the classical Lebesgue space $L^p(\rn)$, via replacing the
constant exponent $p$ by the exponent function $p(\cdot):\ \rn\to(0,\fz)$, which consists of all
measurable functions $f$ such that, for some $\lz\in(0,\fz)$,
$$\int_\rn[|f(x)|/\lz]^{p(x)}\,dx<\fz.$$
The study of variable
Lebesgue spaces can be traced back to Orlicz \cite{or31} and have been the subject of more intensive
study since the early 1990s because of their intrinsic interest for applications in
harmonic analysis \cite{cruz03,cfbook,din04,dhr11,kr91}, in
partial differential equations and variation calculus \cite{am05,hhl08,su09},
in fluid dynamics \cite{am02,rm00} and image processing \cite{cgzl08}.

Moreover, Nakai and Sawano \cite{ns12} introduced the variable Hardy space $H^{p(\cdot)}(\rn)$
with $p(\cdot)$ satisfying the globally log-H\"older continuous condition
and, independently, Cruz-Uribe and Wang \cite{cw14} also investigated the space $H^{p(\cdot)}(\rn)$
with $p(\cdot)$ satisfying some conditions slightly weaker than those used in \cite{ns12}.
Later, Sawano \cite{Sa13} extended the atomic characterization of $H^{p(\cdot)}(\rn)$, which also improves
the corresponding result in \cite{ns12}, Zhuo et al. \cite{zyl16} established
characterizations of $H^{p(\cdot)}(\rn)$ via intrinsic square functions
and Yang et al. \cite{yzn16} characterized the space $H^{p(\cdot)}(\rn)$ by means of
Riesz transforms. Furthermore, in \cite{yyyz16}, Yan et al. introduced the variable
weak Hardy space $W\!H^{p(\cdot)}(\rn)$ with $p(\cdot)$
satisfying the globally log-H\"older continuous condition and proved the boundedness of
convolutional $\delta$-type Calder\'on-Zygmund operators from $H^{p(\cdot)}(\rn)$
to $W\!H^{p(\cdot)}(\rn)$ including the critical case $p_-=\frac n{n+\delta}$.

The main purpose of this article is to introduce and investigate the variable weak Hardy space associated with
operator $L$ on $\rn$, denoted by $W\!H_L^{p(\cdot)}(\rn)$.

Recall that function spaces (especially Hardy spaces) associated
with various operators have been inspired great interests in recent years;
see, for example, \cite{adm05,bckyy13a,bckyy13,dl13,dy05,dy051,hm09,hmm11,jy10,yz16,yzz17,zy15}.
Particularly, using the Lusin area function associated with operator,
Auscher et al. \cite{adm05} initially introduced
the Hardy space $H_L^1(\rn)$ associated with an operator $L$ whose heat kernel
has a pointwise Gaussian upper bound and established
its molecular characterization. Based on this, Duong and Yan \cite{dy05,dy051}
introduced BMO-type spaces associated with $L$
and proved that they are dual spaces of $H_L^1(\rn)$.
Yan \cite{yan08} further generalized these results to the Hardy spaces
$H_L^p(\rn)$ with $p\in(0,1]$ but close to $1$ and their dual spaces.
Moreover, Jiang et al. \cite{jyz09} investigated the Orlicz-Hardy space
and its dual space associated with such an operator $L$. Later, Hardy spaces associated with
operators satisfying the weaker condition, the so-called Davies-Gaffney type estimates,
were studied in \cite{bckyy13a,bckyy13,cchy15,hlmmy,hm09,hmm11} and their references.
In particular, Cao et al. \cite{cchy15} introduced and investigated weak Hardy spaces
$W\!H_L^p(\rn)$ associated with operators satisfying $k$-Davies-Gaffney estimates.

Very recently, via mixing up the concepts of variable function spaces
and functions spaces associated with operators, the real-variable theory of variable Hardy spaces
associated with
operators attracts a lot of attention; see \cite{abdr17,yyz16,yzz17,yz16,zy15}.
More precisely, when $p(\cdot):\ \rn\to(0,1]$ is variable exponent function satisfies
the globally log-H\"older continuous condition,
in \cite{yz16}, the authors first studied variable Hardy spaces associated with
operators $L$ on $\rn$, denoted by $H_L^{p(\cdot)}(\rn)$, where $L$ is a linear operator
on $L^2(\rn)$ and generates an analytic semigroup $\{e^{-tL}\}_{t>0}$ with kernels having
pointwise upper bounds. The molecular characterization and the dual space of $H_L^{p(\cdot)}(\rn)$
were also established in \cite{yz16}. Under an additional assumption that $L$ is non-negative
self-adjoint, the atomic characterization and those characterizations in terms of
maximal functions, including non-tangential maximal functions
and radial maximal functions, were obtained in \cite{zy15}. Moreover, variable Hardy spaces
associated with operators satisfying Davies-Gaffney estimates were introduced and investigated
in \cite{yzz17}, and local Hardy spaces with variable exponents associated to non-negative
self-adjoint operators satisfying Gaussian estimates were studied in \cite{abdr17}.

Motivated by the above results, especially by \cite{cchy15,yzz17,yyyz16}, it is the main target of
this article to establish a real-variable theory of variable weak Hardy spaces associated to
a class of differential operators
and study their applications. Precisely, let $p(\cdot):\ \rn\to(0,1]$ be a variable exponent
satisfying the globally log-H\"older continuous condition [see \eqref{elog} and \eqref{edecay} below],
and $L$ a one-to-one operator of type $\omega$ in $L^2(\rn)$,
with $\omega\in(0,\pi/2)$, which has a bounded holomorphic functional calculus and satisfies
the Davies-Gaffney estimates, namely, Assumptions \ref{as-a} and \ref{as-b} below.
Then we introduce the variable weak Hardy space $W\!H_L^{p(\cdot)}(\rn)$
(see Definition \ref{d-4-9} below). By first obtaining the atomic decomposition of variable
weak tent space, we establish the molecular characterization of $W\!H_L^{p(\cdot)}(\rn)$, which
is further used to obtain its atomic characterization when $L$ is non-negative self-adjoint.
Moreover, when $L$ satisfies Gaussian upper bound estimates
(see Remark \ref{r-3-30}(ii) below), we characterize the space $W\!H_L^{p(\cdot)}(\rn)$ by means of
non-negative maximal functions. In particular, when $L$ is a second-order divergence form
elliptic operator with complex bounded measurable coefficients, namely,
$L:=-{\rm div}(A\nabla)$
[see \eqref{eq op} below for its definition], the boundedness of the Riesz
transform $\nabla L^{-1/2}$ from $W\!H_L^{p(\cdot)}(\rn)$ to the variable weak Hardy
space $W\!H^{p(\cdot)}(\rn)$ is established by using its molecular characterization.

This article is organized as follows.

In Section \ref{s-2}, we first describe Assumptions \ref{as-a} and \ref{as-b}
imposed on the considered operator $L$ of the present article.
Then we recall some notation and notions on variable (weak) Lebesgue spaces and
introduce the definition of the variable weak Hardy space $W\!H_L^{p(\cdot)}(\rn)$
via the square function of the heat semigroup generated by $L$.

In Section \ref{s-3}, we introduce the variable weak tent space $W\!T^{p(\cdot)}(\rr_+^{n+1})$
and then establish its atomic decomposition (see Theorem \ref{t-WT-atom} below)
via the Whitney-type covering lemma and the
Fefferman-Stein vector-valued inequality of the Hardy-Littlewood maximal operator $\cm$ on the variable Lebesgue
space $\lv$ (see Lemma \ref{l-2-1x} below).
We point out that, in the atomic decomposition of each element in $W\!T^{p(\cdot)}(\rr_+^{n+1})$,
an explicit relation between the supports of $T^{p(\cdot)}$-atoms and
the corresponding coefficients is obtained, which
plays a key role in establishing the weak molecular and atomic characterizations
of $W\!H_L^{p(\cdot)}(\rn)$ in Sections \ref{s-3} and \ref{s-4}.

Section \ref{s-4} is devoted to the weak molecular characterization of
$W\!H_L^{p(\cdot)}(\rn)$ (see Theorem \ref{t-mol} below),
which is an immediate consequence of Propositions \ref{p-3-19x} and \ref{p-3-28} below.
In the proof of the inclusion of the weak molecular Hardy space $\mhp$
into $W\!H_L^{p(\cdot)}(\rn)$ (see Propositions \ref{p-3-19x}),
we make full use of a key lemma obtained by Sawano in \cite[Lemma 4.1]{Sa13} (also restated as
in Lemma \ref{l-1-30} below), which reduces some estimates on $L^{p(\cdot)}(\rn)$ norms
for some series of functions into dealing with $L^q(\rn)$ norms for the corresponding functions,
and also borrow some ideas from the proof of \cite[Theorem 4.4]{yyyz16}.

In Section \ref{s-5}, we establish the weak atomic characterization of
$W\!H_L^{p(\cdot)}(\rn)$ (see Theorem \ref{t-1-29} below), under an additional
assumption that $L$ is non-negative self-adjoint. Indeed,
by some arguments similar to those used in the proof of Proposition \ref{p-3-19x},
we show that the weak atomic Hardy space $\wha$ is a subspace of $W\!H_L^{p(\cdot)}(\rn)$.
The converse inclusion depends on the bounded holomorphic functional calculi and
the operator $\Pi_{\Phi,L}$ [see \eqref{4-26x} below], which maps a
$T^{p(\cdot)}(\rr_+^{n+1})$-atom into a $(p(\cdot),2,M)$-atom (see Lemma \ref{l-3-18a} below).

In particular, in Section \ref{s-6}, when $L$ is a non-negative self-adjoint linear operator
on $L^2(\rn)$ and satisfies the Gaussian upper bound estimates
[see Remark \ref{r-3-30}(ii) for more details], we characterize
$W\!H_L^{p(\cdot)}(\rn)$ in terms of non-tangential maximal functions
(see Corollary \ref{c-4-25} below).
We first establish non-tangential maximal function characterizations of
the weak atomic Hardy space $\wha$ by some arguments similar to those used in the proof
of \cite[Theorem 1.11]{zy15}. Then Corollary \ref{c-4-25}
is an immediate consequence of Theorems \ref{t-1-29} and \ref{t-3-16}.

As an application of the molecular decomposition of $W\!H_L^{p(\cdot)}(\rn)$ obtained in
Proposition \ref{p-3-28}, in Section \ref{s-7}, when $L$ is a second-order divergence form
elliptic operator [see Remark \ref{r-3-30}(i) below], we prove that the Riesz transform
$\nabla L^{-1/2}$ is bounded from $W\!H_L^{p(\cdot)}(\rn)$ to
the variable weak Hardy $W\!H^{p(\cdot)}(\rn)$ introduced in \cite{yyyz16}
(see Theorem \ref{t-4-2} below).

We end this section by making some conventions on notation. Throughout this article,
we denote by $C$ a positive constant which is independent of the main parameters,
but it may vary from line to line.
We also use $C_{(\az, \bz,\ldots)}$ to denote a positive constant depending on
the parameters $\az$, $\bz$, $\ldots$.
The \emph{symbol $f\ls g$} means that $f\le Cg$.
If $f\ls g$ and $g\ls f$, then we write $f\sim g$.
For any measurable subset $E$ of $\rn$, we denote by $E^\com$ the \emph{set $\rn\bh E$}
and by $\chi_E$ the characteristic function of $E$.
For any $a\in\mathbb{R}$, the \emph{symbol} $\lfloor a\rfloor$
denotes the largest integer $m$ such that $m\le a$.
Let $\nn:=\{1,\,2,\,\ldots\}$ and $\zz_+:=\nn\cup\{0\}$.
Let $\rnn:=\rn\times(0,\fz)$. For any $\az\in(0,\,\fz)$ and $x\in\rn$, define
\begin{equation}\label{eq bgz}
\bgz_\az(x):=\{(y,\,t)\in\rnn:\ |y-x|<\az t\}.
\end{equation}
If $\az=1$, we simply write $\bgz(x)$ instead of $\bgz_\az(x)$.

For any ball $B:=B(x_B,r_B)\st\rn$ with $x_B\in\rn$ and
$r_B\in(0,\,\fz)$, $\az\in(0,\fz)$ and $j\in\nn$,
we let $\az B:=B(x_B,\az r_B)$,
\begin{align}\label{eq ujb}
U_0(B):=B\ \ \ \text{and}\ \ \ U_j(B):=(2^jB)\setminus (2^{j-1}B).
\end{align}
For any $p\in[1,\,\fz]$, $p'$ denotes its conjugate number,
namely, $1/p+1/p'=1$.

Let $\cs(\rn)$ be the \emph{space of all Schwartz functions}, equipped with the well-known
topology determined by a countable family of seminorms, and $\cs'(\rn)$
the \emph{space of all Schwartz distributions}, equipped with the weak-$\ast$ topology.
For any $r\in(0,\,\fz)$, denote by $L^r_{\loc}(\rn)$ the set of all
\emph{locally $r$-integrable functions} on $\rn$ and, for any measurable
set $E\st\rn$, let $L^r(E)$ be the set of all measurable functions $f$
on $E$ such that $\|f\|_{L^r(E)}:=[\int_E |f(x)|^r\,dx]^{1/r}<\fz.$

\section{The weak Hardy space $W\!H_L^{p(\cdot)}(\rn)$\label{s-2}}
In this section, we first make two assumptions on considered operators $L$,
which are used through the whole
article and then introduce the weak Hardy space associated with the operator $L$
after recalling some notions about the (weak) variable Lebesgue spaces on
$\rn$.

\subsection{Two assumptions of operators $L$}

Before giving the assumptions on operators $L$ studied in this article,
we first recall some knowledge about
bounded holomorphic functional calculi introduced by McIntosh \cite{m86}
(see also \cite{adm96}).

Let $\omega\in[0,\,\pi)$. The \emph{closed} and the \emph{open $\omega$ sectors},
$S_{\oz}$ and $S_{\oz}^0$, are defined, respectively, by setting
\begin{equation*}
S_{\oz}:=\{z\in\mathbb{C}:\ |\arg z|\le\oz\}\cup\{0\}
\quad\mathrm{and}\quad
S^0_{\oz}:=\{z\in\mathbb{C}\setminus\{0\}:\ |\arg z|<\oz\}.
\end{equation*}
A closed and densely defined operator $T$ in $L^2(\rn)$ is said to be of \emph{type $\oz$} if
\begin{enumerate}
\item[(i)] the spectrum $\sigma(T)$ of $T$ is contained in $S_\oz$.

\item[(ii)] for any $\tz\in(\oz,\,\pi)$, there exists a positive constant $C_{(\tz)}$
such that, for any $z\in \mathbb{C}\setminus S_\tz$,
$$|z|\lf\|(zI-T)^{-1}\r\|_{\mathcal{L}(L^2(\rn))}\le C_{(\tz)},$$
here and hereafter, $\mathcal{L}(L^2(\rn))$ denotes the set of all continuous
linear operators from $L^2(\rn)$ to itself and,
for any $S\in\mathcal{L}(L^2(\rn))$,
the operator norm of $S$ is denoted by $\|S\|_{\mathcal{L}(L^2(\rn))}$.
\end{enumerate}

For any $\mu\in(0,\,\pi)$, define
$$H_\fz(S_\mu^0):=\lf\{f:\ S_\mu^0\to\mathbb{C}\ \text{is holomorphic and}\
\|f\|_{L^\fz(S_\mu^0)}<\fz\r\}$$
and
\begin{align*}
\Psi(S_\mu^0):=\lf\{f\in H_\fz(S_\mu^0):\ \exists\,\az,\,C\in(0,\,\fz)\ \text{such that}\
|f(z)|\le \frac{C|z|^\az}{1+|z|^{2\az}},\ \forall z\in S_\mu^0\r\}.
\end{align*}

For any $\oz\in[0,\,\pi)$, let $T$ be a one-to-one operator of type $\oz$ in $L^2(\rn)$.
For any $\psi\in\Psi(S_\mu^0)$ with $\mu\in(\oz,\,\pi)$,
the operator $\psi(T)\in\mathcal{L}(L^2(\rn))$ is defined by setting
\begin{equation}\label{eq function}
\psi(T):=\int_\gamma\psi(\xi)(\xi I-T)^{-1}\,d\xi,
\end{equation}
where $\gamma:=\{re^{i\nu}:\ r\in(0,\,\fz)\}\cup\{re^{-i\nu}:\ r\in(0,\,\fz)\}$,
$\nu\in(\omega,\,\mu)$, is a curve consisting of two rays parameterized anti-clockwise.
It is easy to see that the integral in \eqref{eq function} is absolutely
convergent in $L^2(\rn)$ and the definition of $\psi(T)$ is independent of
the choice of $\nu\in(\oz,\,\mu)$ (see \cite[Lecture 2]{adm96}).
It is well known that the above holomorphic functional calculus defined on
$\Psi(S_\mu^0)$ can be extended to $H_\fz(S_\mu^0)$ by a limiting procedure
(see \cite{m86}).
Let $0\le\omega<\mu<\pi$. Recall that the operator $T$ is said to have a
\emph{bounded holomorphic functional calculus} in $L^2(\rn)$ if
there exists a positive constant $C_{(\omega,\mu)}$, depending on $\omega$ and $\mu$,
such that, for any $\psi\in H_\fz(S_\mu^0)$,
\begin{equation}\label{e2.1x}
\|\psi(T)\|_{\mathcal{L}(L^2(\rn))}\le C_{(\omega,\mu)}\|\psi\|_{L^\fz(S_\mu^0)}.
\end{equation}
By \cite[Theorem F]{adm96}, we know that, if \eqref{e2.1x} holds true for some
$\mu\in (\omega,\pi)$, then it also holds true for any $\mu\in (\omega,\pi)$.

We now make the following two assumptions on the operator $L$.

\begin{assumption}\label{as-a}
$L$ is a one-to-one operator of type $\omega$ in $L^2(\rn)$, with $\oz\in[0,\,\pi/2)$,
and has a bounded holomorphic functional calculus.
\end{assumption}

\begin{assumption}\label{as-b}
The semigroup $\{e^{-tL}\}_{t>0}$ generated by $L$
satisfies the \emph{Davies-Gaffney estimates},
namely, there exist positive constants $C$ and $c$ such that, for any closed subsets
$E$ and $F$ of $\rn$ and $f\in L^2(\rn)$ with $\supp f\st E$,
\begin{align}\label{eq-dg}
\lf\|e^{-tL}(f)\r\|_{L^2(F)}\le Ce^{-c\frac{[\dist(E,\,F)]^2}{t}}\|f\|_{L^2(E)}.
\end{align}
Here and hereafter, for any subsets $E$ and $F$ of $\rn$,
$$\dist(E,\,F):=\inf\{|x-y|:\ x\in E,\,y\in F\}.$$
\end{assumption}

\begin{rem}\label{r-2-1}
\begin{enumerate}
\item[(i)] Let $T$ be a one-to-one operator of type $\omega$ in $L^2(\rn)$ with $\oz\in[0,\,\pi/2)$.
Then it follows from \cite[Theorem 1.45]{ou05}
that $T$ generates a bounded holomorphic semigroup $\{e^{-zT}\}_{z\in S^0_{\pi/2-\oz}}$
on the open sector $S^0_{\pi/2-\oz}$.

\item[(ii)] Let $L$ be an operator satisfying Assumptions \ref{as-a} and \ref{as-b}.
Then, for any $k\in\zz_+$, the family $\{(tL)^ke^{-tL}\}_{t>0}$ of operators satisfies the
Davies-Gaffney estimates \eqref{eq-dg} (see, for example, \cite[Remrk 2.5(i)]{yzz17}).
In particular, for any $k\in\zz_+$ and $t\in(0,\fz)$, the operator $(tL)^ke^{-tL}$ is bounded
on $L^2(\rn)$.
\end{enumerate}
\end{rem}

\begin{rem}\label{r-3-30}
Following \cite[Remark 2.6]{yzz17}, examples of operators satisfying
Assumptions \ref{as-a} and \ref{as-b} include:
\begin{enumerate}
\item[(i)] the second-order divergence form elliptic operator
with complex bounded coefficients as in \cite{hm09,hmm11}.
Recall that a matrix $A(x):=(A_{ij}(x))_{i,j=1}^n$ of complex-valued measurable functions on $\rn$
is said to satisfy the \emph{elliptic condition} if there exist positive constants
$\lz\le \Lambda$ such that, for almost every $x\in\rn$ and any $\xi,\,\eta\in\cc^n$,
$$\lz|\xi|^2\le\Re\langle A(x)\xi,\,\xi\rangle\ \ \ \text{and}\ \ \
|\langle A(x)\xi,\,\eta\rangle|\le\Lambda|\xi||\eta|,$$
where $\langle\cdot,\,\cdot\rangle$ denotes the \emph{inner product} in $\cc^n$ and
$\Re\xi$ denotes the \emph{real part} of $\xi$.
For such a matrix $A(x)$, the associated \emph{second-order divergence form elliptic operator $L$}
is defined by setting, for any $f\in D(L)$,
\begin{equation}\label{eq op}
Lf:=-{\rm div}(A\nabla f),
\end{equation}
which is interpreted in the weak sense via a sesquilinear form.
Here and hereafter, $D(L)$ denotes the domain of $L$.

\item[(ii)] the one-to-one non-negative self-adjoint operator $L$
having the \emph{Gaussian upper bounds}, namely,
there exist positive constants $C$ and $c$ such that, for any $t\in(0,\,\fz)$ and $x,\,y\in\rn$,
\begin{equation*}
|p_t(x,\,y)|\le\frac{C}{t^{n/2}}\exp\lf(-c\frac{|x-y|^2}{t}\r),
\end{equation*}
where $p_t$ denotes the kernel of $e^{-tL}$.

\item[(iii)] the \emph{Schr\"{o}dinger operator} $-\Delta+V$ on $\rn$ with the non-negative
potential $V\in L^1_{\rm loc}(\rn)$ which is not identically zero.
\end{enumerate}
\end{rem}

\subsection{The definition of variable weak Hardy spaces $W\!H_L^{p(\cdot)}(\rn)$}

A measurable function $p(\cdot):\ \rn\to[0,\fz)$ is called a
\emph{variable exponent}.
Denote by $\cp(\rn)$ the \emph{collection of all variable exponents}
$p(\cdot)$ satisfying
\begin{align}\label{2.1x}
0<p_-:=\mathop\mathrm{ess\,inf}_{x\in \rn}p(x)\le
\mathop\mathrm{ess\,sup}_{x\in \rn}p(x)=:p_+<\fz.
\end{align}

For any $p(\cdot)\in\cp(\rn)$, the \emph{variable Lebesgue space} $\lv$ is defined to be the
set of all measurable functions $f$ such that, for some $\lz\in(0,\fz)$,
$\int_\rn[|f(x)|/\lz]^{p(x)}\,dx<\fz$,
equipped with the
\emph{Luxemburg} (also known as the \emph{Luxemburg-Nakano})
\emph{quasi-norm}
\begin{equation*}
\|f\|_{\lv}:=\inf\lf\{\lz\in(0,\fz):\ \int_\rn \lf[\frac{|f(x)|}{\lz}\r]^{p(x)}\le1\r\}.
\end{equation*}

\begin{rem}\label{r-vlp}
 Let $p(\cdot)\in\cp(\rn)$.

\begin{enumerate}
\item[(i)] It is easy to see that, for any $s\in (0,\fz)$ and $f\in\lv$,
$$\lf\||f|^s\r\|_{\lv}=\|f\|_{L^{sp(\cdot)}(\rn)}^s.$$
Moreover, for any $\lz\in{\mathbb C}$ and $f,\ g\in\lv$,
$\|\lz f\|_{\lv}=|\lz|\|f\|_{\lv}$ and
$$\|f+g\|_{\lv}^{\underline{p}}\le \|f\|_{\lv}^{\underline{p}}
+\|g\|_{\lv}^{\underline{p}},$$
here and hereafter,
\begin{align}\label{2.1y}
\underline{p}:=\min\{p_-,1\}
\end{align}
with $p_-$ as in \eqref{2.1x}.
Particularly, when $p_-\in[1,\fz)$,
$\lv$ is a Banach space (see \cite[Theorem 3.2.7]{dhr11}).

\item[(ii)] If there exist $\delta,\,c\in(0,\fz)$ such that  $\int_\rn[|f(x)|/\delta]^{p(x)}\,dx\le c$, then it is easy
to see that $\|f\|_{\lv}\le C\delta$, where $C$ is a positive constant
independent of $\delta$, but depending on $p_-$ (or $p_+$) and $c$.
\end{enumerate}
\end{rem}

A function $p(\cdot)\in\cp(\rn)$ is said to satisfy the
\emph{globally log-H\"older continuous condition}, denoted by $p(\cdot)\in C^{\log}(\rn)$,
if there exist positive constants $C_{\log}(p)$ and $C_\fz$, and
$p_\fz\in\rr$ such that, for any $x,\ y\in\rn$,
\begin{equation}\label{elog}
|p(x)-p(y)|\le \frac{C_{\log}(p)}{\log(e+1/|x-y|)}
\end{equation}
and
\begin{equation}\label{edecay}
|p(x)-p_\fz|\le \frac{C_\fz}{\log(e+|x|)}.
\end{equation}

\begin{defn}
Let $p(\cdot)\in\cp(\rn)$.
The \emph{variable weak Lebesgue space} $W\!L^{p(\cdot)}(\rn)$
is defined to be the set of all measurable functions $f$ such that
\begin{align*}
\|f\|_{WL^{p(\cdot)}(\rn)}:=\sup_{\az\in(0,\fz)}\az\lf
\|\chi_{\{x\in\rn:\ |f(x)|>\az\}}\r\|_{L^{p(\cdot)}(\rn)}<\fz.
\end{align*}
\end{defn}

\begin{rem}\label{r-3-18}
Let $p(\cdot)\in \cp(\rn)$.
\begin{enumerate}
\item[(i)] Then $\|\cdot\|_{\wlv}$
defines a quasi-norm on $\wlv$, namely,
$\|f\|_{\wlv}=0$ if and only if $f=0$ almost everywhere;
for any $\lz\in\mathbb C$ and
$f\in\wlv$, $\|\lz f\|_{\wlv}=|\lz|\|f\|_{\wlv}$ and,
for any $f,\ g\in\wlv$,
$$\|f+g\|_{\wlv}^{\underline{p}}\le 2^{\underline{p}}
\lf[\|f\|_{\wlv}^{\underline{p}}+\|g\|_{\wlv}^{\underline{p}}\r],$$
where $\underline{p}$ is as in \eqref{2.1y}.

\item[(ii)] By the Aoki-Rolewicz theorem (see \cite{Ao42,Ro57}
and also \cite[Exercise 1.4.6]{g09-1}),
we know that there exists a positive constant $v\in(0,1)$
such that, for any $R\in\nn$ and $\{f_j\}_{j=1}^R$,
$$\lf\|\sum_{j=1}^R |f_j|\r\|_{\wlv}^{v}
\le 4\lf[\sum_{j=1}^R \|f_j\|_{\wlv}^{v}\r].$$
\end{enumerate}
\end{rem}

Assume that the operator $L$ satisfies Assumptions \ref{as-a} and \ref{as-b}. For any $k\in\nn$,
the \emph{square function $S_{L,\,k}$ associated with $L$} is defined by setting,
for any $f\in L^2(\rn)$ and $x\in\rn$,
\begin{equation*}
S_{L,\,k}(f)(x):=\lf[\iint_{\bgz(x)}\lf|(t^2L)^ke^{-t^2L}(f)(y)\r|^2\,\frac{dy\,dt}{t^{n+1}}\r]^{1/2},
\end{equation*}
where $\Gamma(x):=\{(y,t)\in\rr_+^{n+1}:\ |y-x|<t\}$.
In particular, when $k=1$, we write $S_L$ instead of $S_{L,\,k}$.
Notice that, for any $k\in\nn$, $S_{L,\,k}$ is bounded on $L^2(\rn)$.
Indeed, by the Fubini theorem, we know that, for any $f\in L^2(\rn)$,
\begin{align}\label{eq-1}
\int_\rn[S_{L,\,k}(f)(x)]^2\,dx
&=\int_\rn\int_0^\fz\int_{|y-x|<t}\lf|(t^2L)^ke^{-t^2L}(f)(y)\r|^2\,\frac{dy\,dt}{t^{n+1}}\,dx\\
&=\int_\rn\int_0^\fz\lf|(t^2L)^ke^{-t^2L}(f)(y)\r|^2\,\frac{dt}{t}\,dy
\ls\|f\|_{L^2(\rn)}^2,\noz
\end{align}
where the last step in \eqref{eq-1} is from \cite[Theorem F]{adm96}
(see also \cite[(4.1)]{hlmmy}).

Now we introduce the variable weak Hardy space $W\!H_L^{p(\cdot)}(\rn)$.

\begin{defn}\label{d-4-9}
Let $p(\cdot)\in C^{\log}(\rn)$ with $p_+\in(0,1]$ and $L$ satisfy Assumptions \ref{as-a} and \ref{as-b}.
A function $f\in L^2(\rn)$ is said to be in $\mathbb{W\!H}_L^{p(\cdot)}(\rn)$ if $S_L(f)\in \wlv$;
moreover, define $\|f\|_{W\!H_L^{p(\cdot)}(\rn)}:=\|S_L(f)\|_{\wlv}$. Then the \emph{variable weak
Hardy space $W\!H_L^{p(\cdot)}(\rn)$, associated to the operator $L$}, is defined to be
the completion of $\mathbb{W\!H}_L^{p(\cdot)}(\rn)$ with respect to the \emph{quasi-norm}
$\|\cdot\|_{W\!H_L^{p(\cdot)}(\rn)}$.
\end{defn}

\begin{rem}
\begin{enumerate}
\item[(i)] When $p(\cdot)\equiv {\rm constant}\in(0,1]$, the space $W\!H_L^{p(\cdot)}(\rn)$
was studied in \cite{cchy15} as a special case. Indeed, in \cite{cchy15}, Cao et al.
assumed that the operator satisfies the $k$-Davies-Gaffney estimates.

\item[(ii)] We point out that, differently from the Hardy space $H^{p(\cdot)}(\rn)$, with
$p(\cdot):\ \rn\to(0,1]$, in which the space $L^2(\rn)\cap H^{p(\cdot)}(\rn)$ is
dense in $H^{p(\cdot)}(\rn)$ (see the proof of \cite[Theorem 4.5]{ns12}),
the space $L^2(\rn)\cap W\!H^{p(\cdot)}(\rn)$ is not dense in the variable weak Hardy space
$W\!H^{p(\cdot)}(\rn)$ even in the constant case (see Fefferman and Soria \cite{fs86} and see also
He \cite{He14}). When $L:=-\Delta$, the variable weak Hardy space
$W\!H_\Delta^{p(\cdot)}(\rn)$ defined as in Definition \ref{d-4-9} coincides with
the space
$$\overline{W\!H^{p(\cdot)}(\rn)\cap L^2(\rn)}^{\|\cdot\|_{W\!H^{p(\cdot)}(\rn)}},$$
namely, the closure of $W\!H^{p(\cdot)}(\rn)\cap L^2(\rn)$ on the
quasi-norm $\|\cdot\|_{W\!H^{p(\cdot)}(\rn)}$ and hence is a proper subspace of
 $W\!H^{p(\cdot)}(\rn)$.
\end{enumerate}
\end{rem}
\section{Variable weak tent spaces $W\!T^{p(\cdot)}(\rr_+^{n+1})$\label{s-3}}

In this section, we first introduce the variable weak tent space $W\!T^{p(\cdot)}(\rr_+^{n+1})$
and then give its atomic decomposition, which is used later to establish atomic and molecular
characterizations of $W\!H_L^{p(\cdot)}(\rn)$.
For any measurable function $f$ on $\rnn$
and $x\in\rn$, define
\begin{equation*}
A(f)(x):=\lf[\iint_{\Gamma(x)}|f(y,\,t)|^2\,dy\,dt\r]^{1/2}.
\end{equation*}
For any $q\in(0,\,\fz)$, the \emph{tent space $T^q(\rnn)$} is defined to be
the space of all measurable functions $f$ such that
$$\|f\|_{T^q(\rnn)}:=\|A(f)\|_{L^q(\rn)}<\fz.$$

\begin{defn}
Let $p(\cdot)\in \cp(\rn)$. The variable weak tent space $W\!T^{p(\cdot)}(\rr_+^{n+1})$ is defined to be
the set of all measurable functions $f$ on $\rr_+^{n+1}$ such that
$$\|f\|_{W\!T^{p(\cdot)}(\rnn)}:=\|A(f)\|_{WL^{p(\cdot)}(\rn)}.$$
\end{defn}

For any open set $O\st\rn$, the \emph{tent} over $O$  is defined by setting
\begin{align*}
\wh{O}:=\lf\{(y,\,t)\in\rnn:\ \dist \lf(y,\,O^\com\r)\geq t\r\}.
\end{align*}

Let $p(\cdot)\in\cp(\rn)$. Recall that a measurable
function $a$ on $\rnn$ is called a \emph{$T^{p(\cdot)}(\rr_+^{n+1})$-atom} if there
exists a ball $B\st\rn$ such that
\begin{enumerate}
\item[(i)] $\supp a\st \wh{B}$;

\item[(ii)] for any $q\in(1,\,\fz)$, $\|a\|_{T^q(\rnn)}\le |B|^{1/q}\|\chi_B\|_{\vp}^{-1}$.
\end{enumerate}

We point out that the notion of $T^{p(\cdot)}(\rr_+^{n+1})$-atoms was first introduced in \cite{zy15}.
For any $p(\cdot)\in\cp(\rn)$ with $0<p_-\le p_+\le 1$,
any sequences $\{\lz_j\}_{j\in\nn}\st\cc$ and $\{B_j\}_{j\in\nn}$ of balls in $\rn$, let
\begin{equation*}
\ca(\{\lz_j\}_{j\in\nn},\,\{B_j\}_{j\in\nn}):=\lf\|\lf\{\sum_{j\in\nn}
\lf[\frac{|\lz_j|\chi_{B_j}}{\|\chi_{B_j}\|_{\vp}}\r]^{p_-}\r\}^{\frac1{p_-}}\r\|_{\vp},
\end{equation*}
where $p_-$ is as in \eqref{2.1x}.

The main result of this section is stated as follows.

\begin{thm}\label{t-WT-atom}
Let $p(\cdot)\in C^{\log}(\rn)$ with $p_+\in(0,1]$. Then, for each $F\in W\!T^{p(\cdot)}(\rnn)$,
there exists a sequence $\{\aij\}_{i\in\zz,j\in\nn}$ of $T^{p(\cdot)}(\rr_+^{n+1})$-atoms
associated, respectively, to the balls $\{B_{i,j}\}_{i\in\zz,j\in\nn}$ such that
\begin{enumerate}
\item[{\rm(i)}] $F=\sum_{i\in\zz,j\in\nn}\lz_{i,j}a_{i,j}$ almost everywhere on $\rnn$, where
$\lz_{i,j}:=2^i\|\chi_{B_{i,j}}\|_{\vp};$

\item[{\rm (ii)}] there exists a positive constant $C$, independent of $F$, such that
 \begin{equation*}
 \sup_{i\in\zz}\ca(\{\lij\}_{j\in\nn},\{\Bij\}_{j\in\nn})\le C\|F\|_{W\!T^{p(\cdot)}(\rr_+^{n+1})};
 \end{equation*}

\item[{\rm(iii)}] there exist $c\in(0,1)$ and $M_0\in\nn$ such that, for any $i\in\zz$,
 $\sum_{j\in\nn}\chi_{c\Bij}\le M_0$.
\end{enumerate}
\end{thm}

To prove Theorem \ref{t-WT-atom}, we need some known facts as follows
(see, for example, \cite{jy10}).
Let $F$ be a closed subset of $\rn$ and $O:=F^\complement$.
Assume that $|O|<\fz$.
For any fixed $\gamma\in(0,1)$, $x\in\rn$ is said to have the
\emph{global $\gamma$-density} with respect to $F$ if, for any $t\in(0,\fz)$,
${|B(x,t)\cap F|}/{|B(x,t)|}\ge\gz$. Denote by $F_\gz^\ast$ the
\emph{set of all such $x$} and let $O_\gz^\ast:=(F_\gz^\ast)^\complement$.
Then $O\subset O_\gamma^\ast$,
$$O_\gamma^\ast=\lf\{x\in\rn:\ \cm(\chi_O)(x)>1-\gamma\r\},$$
$O_\gamma^\ast$ is open and there exists a positive constant $C_{(\gamma)}$,
depending on $\gamma$, such that
$|O_\gamma^\ast|\le C_{(\gamma)}|O|$. Here and hereafter,
$\mathcal M$ denotes the usual Hardy-Littlewood
maximal function, namely, for any $f\in L_{\loc}^1(\rn)$ and $x\in\rn$,
\begin{equation}\label{2-1x}
\mathcal M(f)(x):=\sup_{B\ni x}\frac1{|B|}\int_B|f(y)|\,dy,
\end{equation}
where the supremum is taken over all balls of $\rn$ containing $x$.
For any $\nu\in(0,\fz)$ and $x\in\rn$, denote by $\mathcal R_\nu F$ the
\emph{union of all cones with vertices in $F$}, namely,
$\mathcal R_\nu F:=\bigcup_{x\in F}\Gamma_\nu(x)$.

The following lemma is just \cite[Lemma 3.1]{jy10}.

\begin{lem}\label{l-1-11}
Let $v,\ \eta\in(0,\fz)$. Then there exist positive constants $\gamma\in(0,1)$ and $C$ such that, for
any closed subset $F$ of $\rn$ with $F^\com$ having finite measure and any non-negative measurable
function $H$ on $\rnn$,
$$\int_{\mathcal R_v(F_\gamma^\ast)}H(y,t)t^n\,dy\,dt\le C\int_F\lf\{\int_{\gamma_\eta}
H(y,t)\,dy\,dt\r\}\,dx.$$
\end{lem}

We also need the following well-known Whitney-type covering lemma (see, for example,
\cite[Lemma 2.6]{gly08}).

\begin{lem}\label{l-1-11x}
Let $\Omega\subset \rn$ be an open set and $C_0\in[1,\fz)$ a positive constant. For any $x\in \rn$,
let $r(x):=d(x,\Omega^\com/(2C_0)$. Then there exist two sequences $\{x_i\}_{i\in\nn}$
of points contained in $\Omega$ and $\{r_i\}_{i\in\nn}:=\{r(x_i)\}_{i\in\nn}$ of positive numbers
such that
\begin{enumerate}
\item[\rm(i)] $\{B(x_i, \frac{r_i}5)\}_{i\in\nn}$ are disjoint;
\item[\rm(ii)] $\Omega=\bigcup_{i\in\nn}B(x_i,r_i)$;
\item[\rm(iii)] for any $i\in\nn$, $B(x_i,C_0r_i)\subset \Omega$;
\item[\rm(iv)]for any $x\in B(x_i,C_0r_i)$, $C_0 r_i\le d(x,\Omega^\com)\le 3C_0r_i$;
\item[\rm(v)] for any $i\in\nn$, there exists a point $x_i^\ast\in \Omega^\com$ such that
$d(x_i^\ast,x_i)<3C_0r_i$;
\item[\rm(vi)] there exists a positive constant $M_0$ such that, for any $x\in\Omega$,
$$\sum_{i\in\nn}\chi_{B(x_i,C_0r_i)}(x)\le M_0.$$
\end{enumerate}
\end{lem}

The following result is just \cite[Lemma 2.6]{zy15}
(For the case when $p_-\in(1,\fz)$, see also \cite[Corollary 3.4]{Iz10}).

\begin{lem}\label{l-2-2}
Let $p(\cdot)\in C^{\log}(\rn)$.
Then there exists a positive constant $C$ such that, for any balls $B_1$, $B_2$
of $\rn$ with $B_1\subset B_2$,
$$C^{-1}\lf(\frac{|B_1|}{|B_2|}\r)^{\frac 1{p_-}}
\le\frac{\lf\|\chi_{B_1}\r\|_{L^{p(\cdot)}(\rn)}}
{\lf\|\chi_{B_2}\r\|_{L^{p(\cdot)}(\rn)}}
\le C\lf(\frac{|B_1|}{|B_2|}\r)^{\frac 1{p_+}}.$$
\end{lem}

The following Fefferman-Stein vector-valued inequality of
the maximal operator $\cm$ on the variable Lebesgue space $\lv$
was obtained in \cite[Corollary 2.1]{cf06}.

\begin{lem}\label{l-2-1x}
Let $r\in(1,\fz)$ and $p(\cdot)\in C^{\log}(\rn)$ satisfy $1<p_-\le p_+<\fz$.
Then there exists a positive
constant $C$ such that, for any
sequences $\{f_j\}_{j=1}^\fz$ of measurable functions,
$$\lf\|\lf\{\sum_{j=1}^\fz
\lf[\cm(f_j)\r]^r\r\}^{1/r}\r\|_{\lv}
\le C\lf\|\lf(\sum_{j=1}
^\fz|f_j|^r\r)^{1/r}\r\|_{\lv},$$
where $\cm$ denotes the Hardy-Littlewood maximal operator as in \eqref{2-1x}.
\end{lem}

\begin{rem}\label{r-1-30}
\begin{enumerate}
\item[(i)] Let $p(\cdot)\in C^{\log}(\rn)$  and $\beta\in[1,\fz)$. Then, by Lemma \ref{l-2-1x}
and the fact that, for any ball $B\subset\rn$ and $r\in(0,\min\{p_-,1\})$,
$\chi_{\beta B}\le\beta^{\frac{n}{r}}[\cm(\chi_B)]^{\frac{1}{r}}$,
we conclude that there exists a positive
constant $C$ such that, for any sequence $\{B_j\}_{j\in\nn}$ of balls of $\rn$,
$$\lf\|\sum_{j\in\nn}\chi_{\beta B_j}\r\|_{\lv}\le
C\beta^{\frac{n}{r}}\lf\|\sum_{j\in\nn}\chi_{B_j}\r\|_{\lv}.$$

\item[(ii)] Let $p(\cdot)\in C^{\log}(\rn)$ with $p_-\in(1,\fz)$.
Then, as a special case of Lemma \ref{l-2-1x},
we know that there exists a positive constant $C$ such that, for any $f\in \vp$,
$$\|\cm(f)\|_\vp\le C \|f\|_\vp.$$
\end{enumerate}
\end{rem}

\begin{rem}\label{r-1-30x}
Let $p(\cdot)\in C^{\log}(\rn)$
and $\{B_j\}_{j\in\nn}$ be a sequence of balls of $\rn$ satisfying that there exist $c\in(0,1]$
and $M_0\in\nn$ such that $\sum_{j\in\nn}\chi_{cB_j}\le M_0$.
Then, by Remark \ref{r-1-30}(i),
we further conclude that
\begin{align*}
\lf\|\lf(\sum_{j\in\nn}\chi_{B_j}\r)^{\frac1{\min\{p_-,1\}}}\r\|_{\lv}
&=\lf\|\sum_{j\in\nn}\chi_{B_j}\r\|_{L^{\frac{p(\cdot)}{\min\{p_-,1\}}}}^{\frac1{\min\{p_-,1\}}}
\ls \lf\|\sum_{j\in\nn}\chi_{cB_j}\r\|_{L^{\frac{p(\cdot)}{\min\{p_-,1\}}}}^{\frac1{\min\{p_-,1\}}}\\
&\sim\lf\|\lf(\sum_{j\in\nn}\chi_{cB_j}\r)^{\frac1{\min\{p_-,1\}}}\r\|_{\lv}\ls
\lf\|\sum_{j\in\nn}\chi_{cB_j}\r\|_{\lv}
\ls \lf\|\sum_{j\in\nn}\chi_{B_j}\r\|_{\lv},
\end{align*}
where the equivalent positive constants are independent of $\{B_j\}_{j\in\nn}$.
\end{rem}

\begin{proof}[Proof of Theorem \ref{t-WT-atom}]
Let $F\in W\!T^{p(\cdot)}(\rr_+^{n+1})$. For any $i\in\zz$, let
$$O_i:=\{x\in\rn:\ A(F)(x)>2^i\}.$$
Then, for any $i\in\zz$, $O_{i+1}\subset O_i$.
Since $F\in W\!T^{p(\cdot)}(\rr_+^{n+1})$, it follows that
\begin{align*}
|O_i|&=\int_{\rn}\lf[\frac{\chi_{O_i}(x)}{\|\chi_{O_i}\|_{\vp}}\r]^{p(x)}
\|\chi_{O_i}\|_{\vp}^{p(x)}\,dx
\le\max\lf\{\|\chi_{O_i}\|_{\vp}^{p_-},\|\chi_{O_i}\|_{\vp}^{p_+}\r\}\\
&\le \max\lf\{2^{-ip_-},2^{-ip_+}\r\}\|F\|_{W\!T^{p(\cdot)}(\rnn)}<\fz.
\end{align*}
Let $\gamma$ be as in Lemma \ref{l-1-11}. In what follows, we denote $(O_i)_\gamma^\ast$ simply by
$O_i^\ast$.
For each $i\in\zz$, using Lemma \ref{l-1-11x}, we obtain a Whitney-type covering
$\{\wBij\}_{j\in\nn}$ of $O_i^\ast$. For any $i\in\zz$ and $j\in\nn$, let
$$\Delta_{i,j}:=\lf[\wBij\times(0,\fz)\r]\cap
\lf(\widehat{O_i^\ast}\backslash \widehat{O_{i+1}^\ast}\r),$$
$$a_{i,j}(y,t):=2^{-i}\|\chi_{\Bij}\|_{\vp}^{-1}F(y,t)\chi_{\Delta_{i,j}}(y,t),\quad \forall\,(y,t)\in\rnn,$$
and $\lij:=2^i\|\chi_{\Bij}\|_{\vp}$, where $\widehat{O_i^\ast}$ denotes the tent over $O_i^\ast$.
By an argument similar to that used in the proof of \cite[Theorem 3.2]{hyy}, we find that
$$\supp F\subset \lf[\lf(\bigcup_{i\in\zz}\widehat{O_i^\ast}\r)\cup E\r],$$
where $E\subset \rnn$ satisfies $\int_E\,\frac{dy\,dt}t=0$. From this, we further deduce that
\begin{equation}\label{1-12x}
F=\sum_{i\in\zz}\sum_{j\in\nn}\lij a_{i,j}
\end{equation}
almost everywhere on $\rnn$.

Next we show that, for each $i\in\zz$ and $j\in\nn$, $a_{i,j}$ is a
$T^{p(\cdot)}(\rr_+^{n+1})$-atom
supported on $$\widehat{3C_0\wBij}=:\Bij:=B(x_{\Bij},r_{\Bij}),$$
where $C_0$ is as in Lemma \ref{l-1-11x}.
Indeed, by Lemma \ref{l-1-11x}(iv) and
the definitions of $\Delta_{i,j}$ and $\widehat{O_i^\ast}$,
we conclude that $\Delta_{i,j}\subset \Bij$, namely,
$\supp \aij\subset \Bij$. Moreover,
for any $(y,t)\in \Delta_{i,j}$, by the fact that $\Delta_{i,j}\subset\widehat{O_i^\ast}$
and Lemma \ref{l-1-11x}(iv) again, we easily find that
$t\le \dist(y, (O_i^\ast)^\com)<r_{\Bij}$.
Then, by this, Lemma \ref{l-1-11}, the H\"older inequality and the fact that
$$\Delta_{i,j}\subset \lf(\widehat{O_{i+1}^\ast}\r)^\complement
=\mathcal R_1\lf((O_{i+1}^\ast)^\com\r),$$
we know that,
for any $H\in T^{q'}(\rnn)$ with $q\in(1,\fz)$ and $\|H\|_{T^{q'}(\rnn)}\le 1$,
\begin{align*}
|\langle a_{i,j},H\rangle|
:=&\lf|\int_{\rnn} a_{i,j}(y,t)H(y,t)\chi_{\Delta_{i,j}}(y,t)\,\frac{dy\,dt}{t}\r|\\
\le &\int_{\mathcal R_1(\rn \backslash O_{i+1}^\ast)}
|a_{i,j}(y,t)H(y,t)|\chi_{\Delta_{i,j}}(y,t)\,\frac{dy\,dt}{t}\\
\ls &\int_{\rn \backslash O_{i+1}^\ast}\lf\{\int_0^{r_{\Bij}}\int_{|y-x|<t}
|a_{i,j}(y,t)H(y,t)|\chi_{\Delta_{i,j}}(y,t)\,\frac{dy\,dt}{t^{n+1}}\r\}\,dx\\
\sim&\int_{(1+3C_0)\wBij\backslash O_{i+1}}\lf\{\int_0^{r_{\Bij}}\int_{|y-x|<t}
|a_{i,j}(y,t)H(y,t)|\chi_{\Delta_{i,j}}(y,t)\,\frac{dy\,dt}{t^{n+1}}\r\}\,dx\\
\ls& \lf\{\int_{(1+3C_0)\wBij\backslash O_{i+1}}|A(\aij)(x)|^q\,dx\r\}^\frac1q
\|A(H)\|_{L^{q'}(\rn)}\\
\ls& 2^{-i}\|\chi_\Bij\|_{\vp}^{-1}\|H\|_{T^{q'}(\rnn)}\lf\{\int_{(1+3C_0)\wBij
\backslash O_{i+1}}|A(F)(x)|^q\,dx\r\}^\frac1q
\ls |\Bij|^\frac1q\|\chi_{\wBij}\|_{\vp}^{-1},
\end{align*}
where we used the fact that $A(F)(y,t)\le 2^{i+1}$ for any $(y,t)\in O_{i+1}^\com$ in the last
inequality.
From this, Lemma \ref{l-2-2} and the dual relation $(T^q(\rnn))^\ast=T^{q'}(\rnn)$ (see \cite{cms85}),
where $(T^q(\rnn))^\ast$ denotes
the dual space of $T^q(\rnn)$, we deduce that, for any $q\in(1,\fz)$,
$$\|\aij\|_{T^q(\rnn)}\ls |\Bij|^\frac1q\|\chi_{\wBij}\|_{\vp}^{-1}
\sim|\Bij|^\frac1q\|\chi_{\Bij}\|_{\vp}^{-1}.$$
 Thus, $\aij$ is a $T^{p(\cdot)}(\rr_+^{n+1})$-atom
associated to the ball $\Bij$, which, combined with \eqref{1-12x}, implies (i).

To show (ii), by Remark \ref{r-1-30}(i) and Lemma \ref{l-1-11x}, we find that, for any $i\in\zz$,
\begin{align*}
\ca(\{\lij\}_{j\in\nn},\{\Bij\}_{j\in\nn})
&=\lf\|\lf\{2^{ip_-}\sum_{j\in\nn}\chi_{\Bij}\r\}^{\frac1{p_-}}\r\|_{\vp}\\
&\ls 2^i\lf\|\lf\{\sum_{j\in\nn}\chi_{\frac1{5}\wBij}\r\}^{\frac1{p_-}}\r\|_{\vp}
\ls 2^i\lf\|\chi_{O_i^\ast}\r\|_{\vp},
\end{align*}
which, together with Remark \ref{r-1-30}(ii) and
the fact that $\chi_{O_i^\ast}\ls [\mathcal M(\chi_{O_i}^r)]^\frac1r$ with $r\in(0,p_-)$,
further implies that
\begin{align*}
\ca(\{\lij\}_{j\in\nn},\{\Bij\}_{j\in\nn})
\ls 2^i\lf\|\mathcal M(\chi_{O_i}^r)\r\|_{L^{\frac{p(\cdot)}r}(\rn)}^\frac1r
\ls 2^i\|\chi_{O_i}\|_{\vp}\ls \|F\|_{W\!T^{p(\cdot)}(\rr_+^{n+1})}.
\end{align*}
Therefore, the conclusion (ii) holds true.

The conclusion of (iii) is just an immediate consequence of Lemma \ref{l-1-11x}(vi). This finishes
the proof of Theorem \ref{t-WT-atom}.
\end{proof}

\section{Molecular characterization of $W\!H_L^{p(\cdot)}(\rn)$\label{s-4}}
In this section, we establish a molecular characterization
of $W\!H_L^{p(\cdot)}(\rn)$ (see Theorem \ref{t-mol}) by using the atomic decomposition
of variable weak tent spaces obtained in Theorem \ref{t-WT-atom}.

\begin{defn}\label{d-1-17}
Assume $M\in\nn$ and $\uc\in(0,\fz)$. A function $m\in L^2(\rn)$ is called a \emph{$(p(\cdot),\,M,\,\varepsilon)_L$-molecule}
if $m\in R(L^M)$ (the range of $L^M$) and there exists a ball $B:= B(x_B,\,r_B)\st\rn$,
with $x_B\in\rn$ and $r_B\in(0,\,\fz)$, such that, for any $k\in\{0,\,\ldots,\,M\}$
and $j\in\zz_+$,
\begin{equation*}
\lf\|(r_B^{-2}L^{-1})^k (m)\r\|_{L^2(\ujb)}\le
2^{-j\varepsilon}|2^j B|^{1/2}\|\chi_B\|^{-1}_{\vp},
\end{equation*}
where $U_j(B)$ is as in \eqref{eq ujb}.
\end{defn}

\begin{rem}\label{r-4-1}
Let $m$ be a $(p(\cdot),\,M,\,\uc)_L$-molecule as in Definition \ref{d-1-17}
associated to the ball $B\st \rn$. If $\uc\in(\frac n2,\fz)$, then it is easy to see that,
for any $k\in\{0,\dots,M\}$,
\begin{align*}
\lf\|(r_B^{-2}L^{-1})^k(m)\r\|_{L^2(\rn)}\le C |B|^{1/2}\|\chi_B\|_{L^{p(\cdot)}(\rn)}^{-1}
\end{align*}
with $C$ being a positive constant independent of $m$, $k$ and $B$.
\end{rem}
\begin{defn}\label{d-m-Hardy}
Let $p(\cdot)\in C^{\log}(\rn)$ with $p_+\in(0,1]$, $M\in\nn$ and $\varepsilon\in(0,\,\fz)$.
Assume that $\{m_{i,j}\}_{i\in\zz,j\in\nn}$ is a family of $(p(\cdot),M,\varepsilon)_L$-molecules
associated, respectively, to balls $\{\Bij\}_{i\in\zz,j\in\nn}$ of $\rn$
and numbers $\{\lij\}_{i\in\zz,j\in\nn}$ satisfying
\begin{enumerate}
\item[(i)] for any $i\in\zz$ and $j\in\nn$, $\lij:=2^i\|\chi_{\Bij}\|_{\vp}$;
\item[(ii)] there exists a positive constant $C$ such that
$$\sup_{i\in\zz}\ca(\{\lij\}_{j\in\nn},\{\Bij\}_{j\in\nn})\le C;$$
\item[(iii)] there exists a positive constant $c\in(0,1]$ such that, for any $x\in\rn$ and $i\in\zz$,
$$\sum_{j\in\nn}\chi_{c\chi_\Bij}(x)\le M_0$$
with $M_0$ being a positive constant independent of $x$
and $i$.
\end{enumerate}
Then, for any $f\in L^2(\rn)$,
$$f=\sum_{i\in\zz,j\in\nn} \lij m_{i,j}$$
is called a
\emph{weak molecular $(p(\cdot),\,M,\,\varepsilon)_L$-representation} of $f$
if the above summation converges in $L^2(\rn)$.
The \emph{weak molecular Hardy space} $\mhp$ is then defined to be the completion of
the space
\begin{align*}
\mol:=\{f:\ f\ \text{has a weak molecular}\
(p(\cdot),\,M,\,\varepsilon)_L\text{-representation}\}
\end{align*}
with respect to the quasi-norm
\begin{align*}
\|f\|_{W\!H_{L,\,M}^{p(\cdot),\,\uc}(\rn)}
&:=\inf\lf\{\sup_{i\in\zz}\ca\lf(\{\lij\}_{j\in\nn},\,\{\Bij\}_{j\in\nn}\r)\r\},
\end{align*}
where the infimum is taken over all weak molecular
$(p(\cdot),\,M,\,\varepsilon)_L$-representations of $f$ as above.
\end{defn}

Now we have the following first main result of this section.

\begin{prop}\label{p-3-19x}
Let $p(\cdot)\in C^{\log}(\rn)$ with $p_+\in(0,1]$ and $L$ satisfy Assumptions \ref{as-a} and
\ref{as-b}. Assume that $\uc\in(\frac n{p_-},\fz)$ and $M\in\nn\cap(\frac n2[\frac1{p_-}-\frac12],\fz)$.
Then there exists a positive constant $C$ such that, for any
$f\in \mathbb{W\!H}_{L,M}^{p(\cdot),\uc}(\rn)$,
$\|f\|_{W\!H_L^{p(\cdot)}(\rn)}\le C\|f\|_{\mhp}$.
\end{prop}

To prove Proposition \ref{p-3-19x}, we need the following useful variant of \cite[Lemma 4.1]{Sa13},
which is just \cite[Lemma 4.5]{yyyz16}.

\begin{lem}\label{l-1-30}
Let $p(\cdot)\in C^{\log}(\rn)$, $q\in(\max\{1,p_+\},\,\fz)$ and $r\in(0,p_-]$.
Then there exists a positive constant $C$ such that, for any sequence $\{B_j\}_{j\in\nn}$
of balls in $\rn$, $\{\lz_j\}_{j\in\nn}\st\mathbb{C}$ and measurable
functions $\{a_j\}_{j\in\nn}$ satisfying
that, for any $j\in\nn$, $\supp a_j\st B_j$ and $\|a_j\|_{L^q(\rn)}\le |B_j|^{1/q}$,
\begin{align*}
\lf\|\lf[\sum_{j=1}^\fz|\lz_j a_j|^{r}\r]^{\frac{1}{r}}\r\|_{\vp}
\le C\lf\|\lf[\sum_{j=1}^\fz|\lz_j \chi_{B_j}|^{r}\r]^{\frac{1}{r}}\r\|_{\vp}.
\end{align*}
\end{lem}

\begin{proof}[Proof of Proposition \ref{p-3-19x}]
Let $f\in \mol$. Then, by its definition, we know that there exists a sequence
$\{\aij\}_{i\in\zz,j\in\nn}$ of $(p(\cdot),M,\uc)$
-molecules, associated to the balls $\{\Bij\}_{i\in\zz,j\in\nn}$ satisfying
Definition \ref{d-m-Hardy}(iii),
such that
\begin{equation}\label{4-21x}
f=\sum_{i\in\zz,j\in\nn}\lij\mij\quad{\rm in}\quad L^2(\rn),
\end{equation}
where $\lij:=2^i\|\chi_\Bij\|_\vp$ for each $i\in\zz$ and $j\in\nn$, and
\begin{equation*}
\sup_{i\in\zz}\ca(\{\lij\}_{j\in\nn},\{\Bij\}_{j\in\nn})\ls \|f\|_{\mhp}.
\end{equation*}
Thus, by \eqref{4-21x} and the fact that $S_L$ is bounded on $L^2(\rn)$, we find that
$$\lim_{N\to\fz}\lf\|S_L(f)-S_L\lf(\sum_{j=1}^N\lij\mij\r)\r\|_{L^2(\rn)}=0,$$
which implies that there exists a subsequence of $S_L(\sum_{j=1}^N\lij\mij)$
(without loss of generality, we may use the same notation as the original sequence)
such that, for almost every $x\in\rn$,
$$S_L(f)(x)=\lim_{N\to\fz}S_L\lf(\sum_{|i|\le N}\sum_{j=1}^N\lij\mij\r)(x).$$
For any given $\az\in(0,\fz)$, let $i_0\in\zz$ be such that $2^{i_0}\le \az<2^{i_0+1}$.
Then, for almost every $x\in\rn$, we have
\begin{align*}
S_L(f)(x)&\le \sum_{i\in\zz,j\in\nn} |\lij|S_L(\mij)(x)\\
&=\lf[\sum_{i=-\fz}^{i_0-1}\sum_{j\in\nn}+
\sum_{i=i_0}^{\fz}\sum_{j\in\nn}\r]|\lij|S_L(\mij)(x)
=:{\rm G}_1(x)+{\rm G}_2(x).
\end{align*}
Moreover, it holds true that
\begin{align}\label{3-29x}
&\|\{x\in\rn:\ S_L(f)(x)>\az\}\|_{\vp}\\
&\hs\ls\lf\|\chi_{\{x\in\rn:\ {\rm G}_1(x)>\az/2\}}\r\|_{\vp}+
\lf\|\chi_{\{x\in\rn:\ {\rm G}_2(x)>\az/2\}}\r\|_{\vp}
=:{\rm G}_{1,1}+{\rm G}_{1,2}.\noz
\end{align}

For ${\rm G}_{1,1}$, let $b\in(0,p_-)$,
$q\in(1,\min\{2,1/b\})$
and $a\in(0,1-\frac1q)$. Then, by the H\"older inequality, we find that, for any $x\in\rn$,
\begin{align*}
{\rm G}_1(x)
&\le \lf(\sum_{i=-\fz}^{i_0-1}2^{iaq'}\r)^\frac1{q'}
\lf\{\sum_{i=-\fz}^{i_0-1}2^{-i aq}\lf[\sum_{j\in\nn}|\lij|S_L(\mij)(x)
\r]^q\r\}^{\frac1q}\\
&=\frac{2^{i_0a}}{(2^{aq}-1)^{1/q}}
\lf\{\sum_{i=-\fz}^{i_0-1}2^{-i aq}\lf[\sum_{j\in\nn}|\lij|S_L(\mij)(x)
\r]^q\r\}^{\frac1q},
\end{align*}
where $q'$ denotes the conjugate exponent of $q$, namely, $\frac1q+\frac1{q'}=1$.
From this, together with the fact that $b\in(0,p_-)$, Remark \ref{r-vlp}(i) and
the well-known inequality that,
for any $d\in(0,1]$ and $\{\theta_j\}_{j\in\nn}\subset \cc$,
\begin{equation}\label{3-4y}
\lf(\sum_{j\in\nn}|\theta_j|\r)^d\le \sum_{j\in\nn}|\theta_j|^d,
\end{equation}
we deduce that
\begin{align}\label{3-29w}
{\rm G}_{1,1}
&\le \lf\|\chi_{\{x\in\rn:\ \frac{2^{i_0a}}{(2^{aq}-1)^{1/q}}
\{\sum_{i=-\fz}^{i_0-1}2^{-i aq}[\sum_{j\in\nn}|\lij|S_L(\mij)(x)
]^q\}^{\frac1q}>2^{i_0-1}\}}\r\|_{\vp}\\
&\ls 2^{-i_0q(1-a)}\lf\|\sum_{i=-\fz}^{i_0-1}2^{-ia q}\lf[
\sum_{j\in\nn}|\lij|S_L(\mij)\r]^q\r\|_{\vp}\noz\\
&\ls 2^{-i_0q(1-a)}\lf\|\sum_{i=-\fz}^{i_0-1}2^{-ia qb}\sum_{k\in\zz_+}\lf[
\sum_{j\in\nn}|\lij|S_L(\mij)\chi_{U_k(\Bij)}\r]^{qb}\r\|_{L^\frac{p(\cdot)}b(\rn)}^\frac1b\noz\\
&\ls 2^{-i_0q(1-a)} \lf[\sum_{i=-\fz}^{i_0-1}\sum_{k\in\zz_+}2^{-iq(a-1)b}
\vphantom{\lf.\|\{\sum_{j\in\nn}\lf[\|\chi_{\Bij}\|_{\vp}
S_L(\mij)\chi_{U_k(\Bij)}\r]^{qb}
\r\}^\frac1b\|_{\vp}^b]^\frac1b}\r.\noz\\
&\quad\hs\times\lf.\lf\|\lf\{\sum_{j\in\nn}\lf[\|\chi_{\Bij}\|_{\vp}
S_L(\mij)\chi_{U_k(\Bij)}\r]^{qb}
\r\}^\frac1b \r\|_{\vp}^b\r]^\frac1b.\noz
\end{align}
By \cite[(3.12)]{yzz17}, we know that there exists a positive constant $C$ such that,
for any $k\in\zz_+$ and $(p(\cdot),M,\uc)_L$-molecule $m$, associated to a ball $B:= B(x_B,\,r_B)$
with $x_B\in\rn$ and $r_B\in(0,\,\fz)$,
\begin{equation*}
\|S_L(m)\|_{L^2(U_k(B))}\le C 2^{-i\uc}|2^kB|^{1/2}\lf\|\chi_B\r\|_{\vp}^{-1}.
\end{equation*}
This further implies that, for any $i\in\zz$ and $j\in\nn$,
\begin{align}\label{1-27x}
\lf\|\lf[\|\chi_\Bij\|_\vp S_L(\mij)\chi_{U_k(\Bij)}\r]^q\r\|_{L^{\frac2q}(\rn)}
\ls 2^{-k\uc q}|2^k\Bij|^\frac q2.
\end{align}
Since $\frac n{\uc q}>p_-$, it follows that we can choose a positive constant
$r$ such that $r\in(\frac n{\uc qb},\frac{p_-}b)$.
By this, \eqref{1-27x}, Lemma \ref{l-1-30} and Remark \ref{r-1-30}(i), we find that
\begin{align}\label{3-29y}
{\rm G}_{1,1}
&\ls 2^{-i_0q(1-a)} \lf[\sum_{i=-\fz}^{i_0-1}\sum_{k\in\zz_+}2^{-iq(a-1)b}2^{-k\uc qb}
\lf\|\sum_{j\in\nn}\chi_{2^k\Bij} \r\|_{L^{\frac{p(\cdot)}b}(\rn)}\r]^{\frac1b}\\
&\ls 2^{-i_0q(1-a)} \lf[\sum_{i=-\fz}^{i_0-1}\sum_{k\in\zz_+}2^{-iq(a-1)b}2^{-k\uc qb}
2^{kn/r}\lf\|\lf(\sum_{j\in\nn}\chi_{\Bij}
\r)^\frac1b \r\|_{\vp}^b\r]^\frac1b\noz\\
&\ls 2^{-i_0q(1-a)} \lf[\sum_{i=-\fz}^{i_0-1}\sum_{k\in\zz_+}2^{-i[q(a-1)+1]b}2^{-k\uc qb}2^{kn/r}
\r]^\frac1b\sup_{i\in\zz}2^i\lf\|\sum_{j\in\nn}\chi_{\Bij}
\r\|_{\vp}\noz\\
&\ls \alpha^{-1}\|f\|_{\mhp},\noz
\end{align}
where, in the last inequality, we used the fact that $a\in(0,1-\frac1q)$.

To estimate G$_{1,2}$, since $\uc\in(\frac n{p_-},\fz)$, it follows that
there exists $s\in(0,1)$ such that
$\frac n{\uc sp_-}<1$. Then, by \eqref{3-4y} and choosing $r_1\in(\frac n{\uc sp_-},1)$, we have
\begin{align}\label{3-28w}
{\rm G}_{1,2}
&\ls \alpha^{r_1}\lf\|\sum_{i=r_0}^\fz\sum_{j\in\nn}
\lf[2^i\|\chi_{\Bij}\|_\vp S_L(\mij)\r]^{r_1}\r\|_\vp\\
&\ls \alpha^{r_1}\lf\|\sum_{i=r_0}^\fz\sum_{j\in\nn}\sum_{k\in\zz_+}
\lf[2^i\|\chi_{\Bij}\|_\vp S_L(\mij)\r]^{r_1p_-}\chi_{U_k(\Bij)}
\r\|_{L^{\frac{p(\cdot)}{p_-}}(\rn)}^{\frac1{p_-}}\noz\\
&\ls \az^{r_1}\lf\{\sum_{i=i_0}^\fz\sum_{k\in\zz_+}2^{ir_1p_-}
\lf\|\lf(\sum_{j\in\nn}\lf[\|\chi_{\Bij}\|_{\vp}
S_L(\mij)\chi_{U_k(\Bij)}\r]^{r_1p_-}
\r)^{\frac1{p_-}}
\r\|_{L^{p(\cdot)}(\rn)}^{p_-}\r\}^{\frac1{p_-}}.\noz
\end{align}
By this, \eqref{1-27x} with $q$ therein replaced by $r_1$, \eqref{3-28w},
Lemma \ref{l-1-30} and Remarks \ref{r-1-30} and \ref{r-1-30x},
we further conclude that, for any $s\in(0,1)$,
\begin{align}\label{3-29z}
{\rm G}_{1,2}
&\ls \az^{r_1}\lf\{\sum_{i=i_0}^\fz\sum_{k\in\zz_+}2^{ir_1p_-}2^{-k\uc r_1p_-}
\lf\|\sum_{j\in\nn}\chi_{2^k\Bij}
\r\|_{L^{\frac{p(\cdot)}{p_-}}(\rn)}\r\}^{\frac1{p_-}}\\
&\ls \az^{r_1}\lf\{\sum_{i=i_0}^\fz\sum_{k\in\zz_+}2^{ir_1p_-}2^{-k\uc r_1p_-}2^{kn/s}
\lf\|\lf[\sum_{j\in\nn}\chi_{\Bij}\r]^{1/p_-}
\r\|_{L^{p(\cdot)}(\rn)}^{p_-}\r\}^{\frac1{p_-}}\noz\\
&\ls \az^{r_1}\lf\{\sum_{i=i_0}^\fz2^{ir_1p_-}
\lf\|\sum_{j\in\nn}\chi_{\Bij}
\r\|_{L^{p(\cdot)}(\rn)}^{p_-}\r\}^{\frac1{p_-}}\noz\\
&\ls \az^{r_1}\lf\{\sum_{i=i_0}^\fz2^{ir_1p_-}2^{-ip_-}\r\}^{\frac1{p_-}}\|f\|_{\mhp}
\ls \az^{-1}\|f\|_{\mhp}.\noz
\end{align}

Finally, combining \eqref{3-29x}, \eqref{3-29y} and \eqref{3-29z}, we obtain
\begin{align*}
\|f\|_{W\!H_L^{p(\cdot)}(\rn)}=\sup_{\az\in(0,\fz)}\az\|\{x\in\rn:\ S_L(f)(x)>\az\}\|_{\vp}
\ls \|f\|_{\mhp}.
\end{align*}
This finishes the proof of Proposition \ref{p-3-19x}.
\end{proof}

The second main result of this section is stated as follows.

\begin{prop}\label{p-3-28}
Let $L$ satisfy Assumptions \ref{as-a} and \ref{as-b}, $p(\cdot)\in C^{\log}(\rn)$ with
$p_+\in(0,1]$,
$M\in\nn$ and $\uc\in(0,\fz)$. If $f\in \whv\cap L^2(\rn)$,
then $f\in \mol$ and $\|f\|_{\mol}\le C\|f\|_{\whv}$ with $C$ being a positive
constant independent of $f$.
\end{prop}

\begin{proof}
Let $f\in W\!H_L^{p(\cdot)}(\rn)\cap L^2(\rn)$.
Then, by its definition and the boundedness of $t^2Le^{-t^2L}$ from $L^2(\rn)$
to $T^2(\rn)$ (see \cite[Theorem F]{adm96}), we find that
$$t^2Le^{-t^2L}(f)\in T^2(\rnn)\cap W\!T^{p(\cdot)}(\rnn).$$
Thus, by Theorem \ref{t-WT-atom},
we know that there exist sequences $\{\lij\}_{i\in\zz,j\in\nn}$ of numbers and
$\{\bij\}_{i\in\zz,j\in\nn}$ of $T^{p(\cdot)}(\rr_+^{n+1})$-atoms associated, respectively,
to balls $\{\Bij\}_{i\in\zz,j\in\nn}$ of $\rn$ such that
$$t^2Le^{-t^2L}(f)=\sum_{i\in\zz,j\in\nn}\lij\bij$$ almost everywhere on $\rnn$,
$\lij=2^i\|\chi_\Bij\|_\vp$
and
\begin{equation}\label{3-28z}
\sup_{i\in\zz}\ca(\{\lij\}_{i\in\zz,j\in\nn},\{\Bij\}_{i\in\zz,j\in\nn})
\ls \|t^2Le^{-t^2L}(f)\|_{W\!T^{p(\cdot)}(\rnn)}\sim \|f\|_{W\!H_L^{p(\cdot)}(\rn)}.
\end{equation}
Moreover, by \cite[Proposition 3.1]{jy10}, we conclude that the decomposition
\begin{equation}\label{3-28y}
t^2Le^{-t^2L}(f)=\sum_{i\in\zz,j\in\nn}\lij\bij
\end{equation}
also holds true in $T^2(\rnn)$. By the bounded holomorphic functional calculi for $L$,
we find that
\begin{align}\label{1-12y}
f=C_{(M)}\int_0^\fz(t^2L)^{M+1}e^{-t^2L}\lf(t^2Le^{-t^2L}(f)\r)\,\frac{dt}t
=:\pi_{M,L}(t^2Le^{-t^2L}(f))
\end{align}
in $L^2(\rn)$, where $C_{(M)}$ is a positive constant such that
$C_{(M)}\int_0^\fz t^{2(M+2)}e^{-2t^2}\,\frac{dt}t=1$ and the operator $\pi_{M,L}$
is defined by setting, for any $F\in T^2(\rnn)$ and $x\in\rn$,
$$\pi_{M,L}(F)(x):=\int_0^\fz (t^2L)^{M+1}e^{-t^2L}(F(\cdot,t))(x)\,\frac{dt}t.$$
Since $\pi_{M,L}$ is bounded from $T^2(\rnn)$ to $L^2(\rn)$
(see, for example, \cite[Proposition 4.5(i)]{bckyy13a}), it follows from
\eqref{3-28y} and \eqref{1-12y} that
\begin{equation}\label{1-16x}
f=C_{(M)}\pi_{\pi,L}\lf(\sum_{i\in\zz,j\in\nn}\lij\bij\r)=C_{(M)}\sum_{i\in\zz,j\in\nn}
\lij\pi_{M,L}(\bij)\quad{\rm in}\quad L^2(\rn).
\end{equation}
Notice that, for any $M\in\nn$, $\uc\in(0,\fz)$, $i\in\zz$ and $j\in\nn$, $\pi_{M,L}(\aij)$
is a $(p(\cdot),M,\uc)_L$-molecule associated to the ball $\Bij$, up to a harmless constant multiple
(see \cite[Lemma 3.11]{yzz17}). Therefore, by this, we further find
that \eqref{1-16x} is a weak molecular $(p(\cdot),M,\uc)_L$-representation of $f$,
which, combined with \eqref{3-28z}, implies that $f\in \mhp$ and
$\|f\|_{\mhp}\ls \|f\|_{W\!H_L^{p(\cdot)}(\rn)}$. This finishes the proof of Proposition
\ref{p-3-28}.
\end{proof}

Combining Propositions \ref{p-3-19x} and \ref{p-3-28}, we immediately conclude the following
molecular characterization of $\whv$, the details being omitted.

\begin{thm}\label{t-mol}
Let $L$ satisfy Assumptions \ref{as-a} and \ref{as-b} and $p(\cdot)\in C^{\log}(\rn)$ with $p_+\in(0,1]$.
Assume $M\in\nn\cap (\frac n2[\frac1{p_-}-\frac12],\fz)$ and $\uc\in(\frac n{p_-},\fz)$. Then
the spaces $\whv$ and $\mhp$ coincide with equivalent quasi-norms.
\end{thm}

\begin{rem}
When $p(\cdot)\equiv {\rm constant}\in(0,1]$, Theorem \ref{t-mol} goes back to
\cite[Theorem 2.21]{cchy15}.
\end{rem}

\section{Atomic characterization of $W\!H_L^{p(\cdot)}(\rn)$\label{s-5}}

In this section, we mainly establish an atomic characterization of the variable weak Hardy space
$W\!H_L^{p(\cdot)}(\rn)$ under an additional condition on $L$, namely,
$L$ is non-negative self-adjoint. We begin with the following notion.

\begin{defn}\label{d-2-1}
Let $p(\cdot)\in C^{\log}(\rn)$,
$L$ be a non-negative self-adjoint operator on $L^2(\rn)$ satisfying Assumptions \ref{as-a} and \ref{as-b},
$M\in\nn$ and $B:=B(x_B,r_B)$ be a ball with $x_B\in \rn$ and $r_B\in(0,\fz)$.
A function $a\in L^2(\rn)$ is called a
$(p(\cdot),2,M)_L$-\emph{atom} associated with some ball $B:=B(x_B,r_B)$
if the following conditions are satisfied:
\begin{enumerate}
\item[(i)] there exists a function $b$ belongs to the domain of $L^M$ such that $a=L^Mb$;
\item[(ii)] for any $\ell\in\{0,\dots,M\}$, $\supp (L^\ell b)\subset B$;
\item[(iii)] for any $\ell\in\{0,\dots,M\}$,
$\|(r_B^2L)^\ell b\|_{L^2(\rn)}\le r_B^{2M}|B|^\frac12\|\chi_B\|_\vp$.
\end{enumerate}

Let $f\in L^2(\rn)$, $\{\lij\}_{i\in\zz,j\in\nn}\subset \cc$ and $\{\aij\}_{i\in\zz,j\in\nn}$
be a sequence of $(p(\cdot),2,M)_L$-atoms associated with balls $\{\Bij\}_{i\in\zz,j\in\nn}$.
Then
$$f=\sum_{i\in\zz}\sum_{j\in\nn}\lij\aij\quad {\rm in}\quad L^2(\rn)$$
is called a \emph{weak atomic $(p(\cdot),2,M)_L$-representation of $f$} if
\begin{enumerate}
\item[(i)] $\lij:=2^i\|\chi_\Bij\|_\vp$;
\item[(ii)] there exists a positive constant $C$, independent of $f$, such that
\begin{equation*}
\sup_{i\in\zz}\ca(\{\lij\}_{j\in\nn},\{\Bij\}_{j\in\nn})\le C;
\end{equation*}
\item[(iii)] there exists a positive constant $c\in(0,1]$ such that, for any $x\in\rn$ and $i\in\zz$,
$\sum_{j\in\nn}\chi_{c\chi_\Bij}(x)\le M_0$ with $M_0$ being a positive constant independent of $x$
and $i$.
\end{enumerate}

The \emph{weak atomic Hardy space} $W\!H_{L,{\rm at},M}^{p(\cdot)}(\rn)$ is defined to be the
completion of the space
\begin{center}
$\mathbb{W\!H}_{L,{\rm at},M}^{p(\cdot)}(\rn):=\lf\{f\in L^2(\rn):\ f\ {\rm has\ a\ weak\
atomic}\ (p(\cdot),2,M)_L-{\rm representation}\r\}$
\end{center}
with respect to the quasi-norm
$$\|f\|_{W\!H_{L,{\rm at},M}^{p(\cdot)}(\rn)}
:=\inf\lf[\sup_{i\in\zz}\ca(\{\lij\}_{j\in\nn},\{\Bij\}_{j\in\nn})\r],$$
where the infimum is taken over all weak atomic $(p(\cdot),2,M)_L$-representations of $f$ as above.
\end{defn}

The weak atomic characterization of the space $W\!H_L^{p(\cdot)}(\rn)$ is stated as follows.

\begin{thm}\label{t-1-29}
Let $p(\cdot)\in C^{\log}(\rn)$ with $p_+\in(0,1]$ and $L$ satisfy Assumptions \ref{as-a}
and \ref{as-b}. If $L$ is non-negative self-adjoint and
$M\in(\frac n2[\frac1{p_-}-\frac12],\fz)\cap \nn$. Then
$$W\!H_L^{p(\cdot)}(\rn)=W\!H_{L,{\rm at},M}^{p(\cdot)}(\rn)$$
with equivalent quasi-norms.
\end{thm}

If $L$ is a non-negative self-adjoint operator on $L^2(\rn)$, then, for any bounded Borel measurable
function $F:\ [0,\fz)\to\cc$, the operator $F(L)$, defined by the formula
\begin{equation*}
F(L):=\int_0^\fz F(\lz)\,dE_{L}(\lz),
\end{equation*}
where $E_L(\lz)$ denotes the spectral decomposition associated with $L$, is bounded on $L^2(\rn)$.
Let $M\in(0,\fz)$, $\phi_0$ be a given even Schwartz function on $\rr$ and $\supp \phi_0\subset(-1,1)$.
Let $\Phi$ be the Fourier transform of
$\phi_0$, namely, for any $\xi\in\rr$, $\Phi(\xi):=\int_\rr\phi_0(x)e^{-ix\cdot\xi}\,dx$.
Then define the operator $\Pi_{\Phi,L}$ by setting,
for any $F\in T^2(\rnn)$ and $x\in\rn$,
\begin{equation}\label{4-26x}
\Pi_{\Phi,L}(F)(x):=C_{(\Phi,M)}\int_0^\fz(t^2L)^{M+1}\Phi(t\sqrt L)(F(\cdot,t))(x)\,\frac{dt}t,
\end{equation}
where $C_{(\Phi,M)}$ is the positive constant such that
$$C_{(\Phi,M)}\int_0^\fz t^{2(M+1)}\Phi(t)t^2e^{-t^2}\,\frac{dt}t=1.$$
Moreover, $\Pi_{\Phi,L}$ is bounded from $T^2(\rnn)$ to $L^2(\rn)$ (see, for example,
\cite[Proposition 4.2(ii)]{jy11}).

By an argument similar to that used in the proof of \cite[Lemma 4.11]{hlmmy}
(see also \cite[Proposition 2.5(i)]{zy15}),
we obtain the following conclusion, the details being omitted.

\begin{lem}\label{l-3-18a}
Let $M\in\nn$, $L$ and $p(\cdot)$ be as in Theorem \ref{t-1-29}. If
$a$ is a $T^{p(\cdot)}(\rr_+^{n+1})$-atom associated with the ball $B$,
then there exists a positive constant $C$, independent of $a$, such that
$C\Pi_{\Phi,L}(a)$ is a $(p(\cdot),2,M)_L$-atom associated with the ball 2B.
\end{lem}

We now turn to the proof of Theorem \ref{t-1-29}.

\begin{proof}[Proof of Theorem \ref{t-1-29}]
To prove Theorem \ref{t-1-29}, it suffices to show that
$$W\!H_L^{p(\cdot)}(\rn)\cap L^2(\rn)=\mathbb{W\!H}_{L,{\rm at},M}^{p(\cdot)}(\rn)$$
with equivalent quasi-norms.

We first prove the inclusion
$$\mathbb{W\!H}_{L,{\rm at},M}^{p(\cdot)}(\rn)\subset W\!H_L^{p(\cdot)}(\rn)\cap L^2(\rn).$$
Let $f\in \mathbb{W\!H}_{L,{\rm at},M}^{p(\cdot)}(\rn)$. Then it follows from its
definition that there exist sequences
$\{\lij\}_{i\in\zz,j\in\nn}$ of numbers and $\{\aij\}_{i\in\zz,j\in\nn}$ of
$(p(\cdot),2,M)_L$-atoms associated, respectively, to balls $\{\Bij\}_{i\in\zz,j\in\nn}$ such that
$\lij=2^i\|\chi_\Bij\|_\vp$,
\begin{equation*}
f=\sum_{i\in\zz,j\in\nn}\lij\aij\quad{\rm in}\quad L^2(\rn)
\end{equation*}
and there exists a constant $c\in(0,1]$ satisfying $\sum_{j\in\nn}\chi_{c\Bij}\ls 1$ with the
implicit positive constant independent of $i\in\zz$.
Moreover, we have
\begin{equation*}
\sup_{i\in\zz}\ca(\{\lij\}_{i\in\zz,j\in\nn},\{\Bij\}_{i\in\zz,j\in\nn})
\ls \|f\|_{W\!H_{L,{\rm at},M}^{p(\cdot)}(\rn)}.
\end{equation*}

For any given $\az\in(0,\fz)$, let $i_0\in\zz$ be such that $2^{i_0}\le \az<2^{i_0+1}$.
Then, by an argument similar to that used in the proof of \eqref{3-29x},
we have
\begin{align*}
&\lf\|\{x\in\rn:\ S_L(f)(x)>\az\}\r\|_{\vp}\\
&\hs\ls\|\chi_{\{x\in\rn:\ {\rm H}_1(x)>\az/2\}}\|_{\vp}+
\|\chi_{\{x\in\rn:\ {\rm H}_2(x)>\az/2\}}\|_{\vp}
=:{\rm H}_{1,1}+{\rm H}_{1,2}.
\end{align*}

We claim that, for any $(p(\cdot),2,M)_L$-atom $a$
associated to some ball $B:=B(x_B,r_B)$, with $x_B\in\rn$ and $r_B\in(0,\fz)$,
it holds true that, for any $\beta\in(0,M)$ and $k\in\zz_+$,
\begin{equation}\label{1-29v}
\|S_L(a)\|_{L^2(U_k(B))}
\ls 2^{-2k\beta}|B|^\frac12\|\chi_B\|_{\vp}^{-1},
\end{equation}
where $U_k(B):=[2^kB]\backslash [2^{k-1}B]$.

To prove this claim, by the boundedness of the operator $S_L$ on $L^2(\rn)$ and the definition of
$(p(\cdot),2,M)_L$-atom,
it suffices to show \eqref{1-29v} for $k\ge 5$. To this end, by the Minkowski inequality,
we first write
\begin{align}\label{1-29y}
\|S_L(a)\|_{L^2(U_k(B))}
&\ls \lf\{\int_{U_k(B)}\int_0^{r_B}\int_{|y-x|<t}
|t^2Le^{-t^2L}(a)(y,t)|^2\,\frac{dy\,dt}{t^{n+1}}\,dx\r\}^\frac12\\
&\quad\hs+\lf\{\int_{U_k(B)}\int_{r_B}^{\frac{\dist(x,B)}{4}}\int_{|y-x|<t}
\cdots\frac{dy\,dt}{t^{n+1}}\,dx\r\}^\frac12\noz\\
&\quad\hs+\lf\{\int_{U_k(B)}\int_{\frac{\dist(x,B)}{4}}^\fz\int_{|y-x|<t}
\cdots\frac{dy\,dt}{t^{n+1}}\,dx\r\}^\frac12\noz\\
&=:{\rm J}_1+{\rm J}_2+{\rm J}_3.\noz
\end{align}

For J$_1$, let $E_1:=\{x\in\rn:\ \dist(x,U_k(B))<r_B\}$. Then $\dist(E_1,B)>2^{k-3}r_B$.
From this, combined with the fact that $\dist(U_k(B),B)>2^{k-2}r_B$ for $k\ge 5$ and
Assumption \ref{as-a},
we deduce that
\begin{align}\label{1-29z}
{\rm J}_1
&\ls \lf\{\int_0^{r_B}\int_{E_1}\int_{|y-x|<t}|t^2Le^{-t^2L}(a)(y,t)|^2
\,dx\frac{dy\,dt}{t^{n+1}}\r\}^\frac12\\
&\sim \lf\{\int_0^{r_B}\lf\|t^2Le^{-t^2L}(a)\r\|_{L^2(E_1)}^2
\,\frac{dt}{t}\r\}^\frac12
\ls  \lf\{\int_0^{r_B} e^{-c\frac{[\dist(E_1,B)]^4}{t^4}}\,\frac{dt}{t}\r\}^\frac12
\|a\|_{L^2(B)}\noz\\
&\ls \lf\{\int_0^{r_B} \frac{t^{4\alpha_1}}{[\dist(E_1,B)]^{4\az_1}}
\,\frac{dt}{t}\r\}^\frac12\|a\|_{L^2(B)}
\ls 2^{-2k\az_1}|B|^{\frac12}\|\chi_B\|_{\vp}^{-1},\noz
\end{align}
where $\az_1$ is chosen such that $\az\in(0,M)$.

For J$_2$, let $E_2:=\{x\in\rn:\ \dist(x,U_k(B))<\frac{\dist(x,B)}4\}$. Then it is easy to see that
$\dist(E_2,B)>2^{k-3}r_B$, which, together with Remark \ref{r-2-1}(ii), further implies that
\begin{align}\label{1-29u}
{\rm J}_2
&\ls \lf\{\int_{r_B}^\fz\int_{E_2}\lf|t^2Le^{-t^2L}(a)(y,t)\r|^2\,\frac{dy\,dt}t\r\}^\frac12\\
&\ls \lf\{\int_{r_B}^\fz \lf\|(t^2L)^{M+1}e^{-t^2L}(L^{-M}a)\r\|_{L^2(E_2)}^2
\,\frac{dt}{t^{4M+1}}\r\}^\frac12\noz\\
&\ls \lf\{\int_{r_B}^\fz \frac{t^{4\az_2}}{[\dist(E_2,B)]^{4\az_2}}\,\frac{dt}{t^{4M+1}}\r\}^\frac12
\|L^{-M}a\|_{L^2(B)}\ls 2^{-2k\az_2}|B|^{\frac12}\|\chi_B\|_{\vp}^{-1},\noz
\end{align}
where $\az_2$ is chosen such that $\az_2\in(0,M)$.

For J$_3$, observe that, for any $x\in U_k(B)$, $\dist(x,B)\ge 2^{k-2}r_B$. It follows, from
Remark \ref{r-2-1}(ii) again, that
\begin{align}\label{1-29w}
{\rm J}_3
&\ls \lf\{\int_{2^{k-2}r_B}^\fz \int_\rn \int_{|y-x|<t} \lf|t^2Le^{-t^2L}(a)(y,t)
\r|^2\,dx\frac{dy\,dt}{t^{n+1}}\r\}^\frac12\\
&\ls \lf\{\int_{2^{k-2}r_B}^\fz \lf\|(t^2L)^{M+1}e^{-t^2L}(L^{-M}a)\r\|_{L^2(\rn)}^2
\,\frac{dt}{t^{4M+1}}\r\}^\frac12\noz\\
&\ls \lf\{\int_{2^{k-2}r_B}^\fz\,\frac{dt}{t^{4M+1}}\r\}^\frac12
\lf\|(L^{-M}a)\r\|_{L^2(B)}\ls 2^{-2kM}|B|^{\frac12}\|\chi_B\|_{\vp}^{-1}.\noz
\end{align}

Combining \eqref{1-29y}, \eqref{1-29z}, \eqref{1-29u} and \eqref{1-29w}, we conclude that
\eqref{1-29v} holds true.

Now we turn to estimate H$_{1,1}$ and H$_{1,2}$.
Let $b\in(0,p_-)$, $q\in(1,\min\{2,\frac1b\})$ and
$a\in(0,1-\frac1q)$. Then, by an argument similar to that used in the proof of \eqref{3-29w},
we obtain
\begin{align}\label{1-30x}
{\rm H}_{1,1}
&\ls 2^{-i_0q(1-a)} \lf[\sum_{i=-\fz}^{i_0-1}\sum_{k\in\zz_+}2^{-iq(a-1)b}
\vpz{\|\lf\{\sum_{j\in\nn}\lf[\|\chi_{\Bij}\|_{\vp}
S_L(\aij)\chi_{U_k(\Bij)}\r]^{qb}\r\}^\frac1b\|_{\vp}^b]^\frac1b}
\vphantom{\lf.
\lf\|\lf\{\sum_{j\in\nn}\lf[\|\chi_{\Bij}\|_{\vp}
S_L(\aij)\chi_{U_k(\Bij)}\r]^{qb}\r\}^\frac1b\r\|_{\vp}^b\r]^\frac1b}\r.\\
&\quad\hs\times\lf.
\lf\|\lf\{\sum_{j\in\nn}\lf[\|\chi_{\Bij}\|_{\vp}
S_L(\aij)\chi_{U_k(\Bij)}\r]^{qb}\r\}^\frac1b\r\|_{\vp}^b\r]^\frac1b.\noz
\end{align}
Moreover, by the above claim \eqref{1-29v}, we find that,
for any $i\in\zz$, $j\in\nn$ and $k\in\zz_+$,
\begin{align}\label{3-4x}
\lf\|[\|\chi_{\Bij}\|_{\vp}
S_L(\aij)\chi_{U_k(\Bij)}]^{q}\r\|_{L^{\frac2q}(\rn)}
\ls 2^{-kq(2\beta+\frac n2)}|2^k\Bij|^{\frac q2}.
\end{align}
We choose $\beta\in(\frac n2(\frac1{p_-}-\frac12),M)$ and $r\in (\frac{p_-}{bq},1)$.
Then $\beta\in (\frac n2(\frac 1{rqb}-\frac12),M)$ and hence
$$q\lf(2\beta+\frac n2\r)>\frac nr.$$
Thus, by this, \eqref{3-4x}, \eqref{1-30x}, Lemma \ref{l-1-30} and Remarks \ref{r-1-30}(i) and
\ref{r-1-30x},
we conclude that
\begin{align*}
{\rm H}_{1,1}
&\ls 2^{-i_0q(1-a)}\lf[\sum_{i=-\fz}^{i_0-1}\sum_{k\in\zz_+}2^{-iq(a-1)b}2^{-kq(2\beta+\frac n2)b}
\lf\|\lf\{\sum_{j\in\nn}\chi_{2^k\Bij}
\r\}^\frac1b\r\|_{\vp}^b\r]^\frac1b\\
&\ls2^{-i_0q(1-a)} \lf[\sum_{i=-\fz}^{i_0-1}\sum_{k\in\zz_+}2^{-iq(a-1)b}2^{-kq(2\beta+\frac n2)b}
2^{kn/r}\lf\|\lf\{\sum_{j\in\nn}\chi_{\Bij}
\r\}^\frac1b\r\|_{\vp}^b\r]^\frac1b\\
&\ls 2^{-i_0q(1-a)} \lf\{\sum_{i=-\fz}^{i_0-1}2^{-iq(a-1)b}2^{-ib}
\lf[\sup_{i\in\zz}2^i\lf\|\sum_{j\in\nn}\chi_{\Bij}\r\|_{\vp}\r]^b\r\}^\frac1b
\ls \az^{-1}\|f\|_{W\!H_{L,{\rm at},M}^{p(\cdot)}(\rn)},
\end{align*}
where, in the last inequality, we used the fact that $a\in(0,1-\frac1q)$.

For H$_{1,2}$, since $\beta\in(\frac n2[\frac1{p_-}-\frac12],M)$, it follows that
there exist $r_2,\ s\in(0,1)$ such that $\beta\in(\frac n2[\frac1{r_2sp_-}-\frac12],M)$.
Then, by \eqref{3-4y}, we find that
\begin{align*}
{\rm H}_{1,2}
&\ls \alpha^{r_2}\lf\|\sum_{i=i_0}^\fz\sum_{j\in\nn}
\lf[2^i\|\chi_{\Bij}\|_\vp S_L(\aij)\r]^{r_2}\r\|_\vp\\
&\ls \alpha^{r_2}\lf\|\sum_{i=i_0}^\fz\sum_{j\in\nn}\sum_{k\in\zz_+}
\lf[2^i\|\chi_{\Bij}\|_\vp S_L(\aij)\r]^{r_2p_-}\chi_{U_k(\Bij)}
\r\|_{L^{\frac{p(\cdot)}{p_-}}(\rn)}^{\frac1{p_-}}\\
&\ls \az^{r_2}\lf\{\sum_{i=i_0}^\fz\sum_{k\in\zz_+}2^{ir_2p_-}
\lf\|\lf(\sum_{j\in\nn}\lf[\|\chi_{\Bij}\|_{\vp}
S_L(\aij)\chi_{U_k(\Bij)}\r]^{r_2p_-}
\r)^{\frac1{p_-}}
\r\|_{L^{p(\cdot)}(\rn)}^{p_-}\r\}^{\frac1{p_-}}.
\end{align*}
From this, together with \eqref{1-29v}, Lemma \ref{l-1-30}, Remark \ref{r-1-30x},
\eqref{3-4x} with $q$ therein replaced by $r_2$ and the fact that
$(2\beta+\frac n2)r_2p_->\frac ns$, we deduce that
\begin{align*}
{\rm H}_{1,2}
&\ls \az^{r_2}\lf\{\sum_{i=i_0}^\fz\sum_{k\in\zz_+}2^{ir_2p_-}2^{-k(2\beta+\frac n2) r_2p_-}
\lf\|\lf[\sum_{j\in\nn}\chi_{2^k\Bij}\r]^{\frac1{p_-}}
\r\|_{\vp}^{p_-}\r\}^{\frac1{p_-}}\\
&\ls \az^{r_2}\lf\{\sum_{i=i_0}^\fz\sum_{k\in\zz_+}2^{ir_2p_-}2^{-k(2\beta+\frac n2) r_2p_-}2^{kn/s}
\lf\|\lf(\sum_{j\in\nn}\chi_{\Bij}\r)^{\frac1{p_-}}
\r\|_{L^{p(\cdot)}(\rn)}^{p_-}\r\}^{\frac1{p_-}}\\
&\ls \az^{r_2}\lf\{\sum_{i=i_0}^\fz2^{ir_2p_-}2^{-ip_-}
\lf[\sup_{i\in\zz}2^i\lf\|\sum_{j\in\nn}\chi_{\Bij}
\r\|_{L^{p(\cdot)}(\rn)}\r]^{p_-}\r\}^{\frac1{p_-}}
\ls \az^{-1}\|f\|_{W\!H_{L,{\rm at},M}^{p(\cdot)}(\rn)},
\end{align*}
where we used the fact that $r_2\in(0,1)$ in the last inequality.

Combining the estimates of H$_{1,1}$ and H$_{1,2}$, we finally conclude that
\begin{equation}\label{3-18w}
\|\{x\in\rn:\ S_L(f)(x)>\az\}\|_{\vp}\\
\ls {\rm H}_{1,1}+{\rm H}_{1,2}\ls \az^{-1}\|f\|_{W\!H_{L,{\rm at},M}^{p(\cdot)}(\rn)},
\end{equation}
which further implies that $\|S_L(f)\|_{W\!L^{p(\cdot)}(\rn)}
\ls \|f\|_{W\!H_{L,{\rm at},M}^{p(\cdot)}(\rn)}$, namely,
$f\in W\!H_L^{p(\cdot)}(\rn)$.

Conversely, we prove that
$W\!H_L^{p(\cdot)}(\rn)\cap L^2(\rn)\subset \mathbb{W\!H}_{L,{\rm at},M}^{p(\cdot)}(\rn)$.
Let $f\in W\!H_L^{p(\cdot)}(\rn)\cap L^2(\rr_+^{n+1})$. Then by its definition and
the boundedness of $t^2Le^{-t^2L}$ from $L^2(\rn)$
to $T^2(\rn)$, we know that
$t^2Le^{-t^2L}(f)\in W\!T^{p(\cdot)}(\rnn)\cap T^{2}(\rnn)$.
From this and Theorem \ref{t-WT-atom}, we deduce
that there exist sequences $\{\lij\}_{i\in\zz,j\in\nn}\subset\cc$ and $\{\bij\}_{i\in\zz,j\in\nn}$
of $T^{p(\cdot)}(\rnn)$-atoms associated, respectively, to balls $\{\Bij\}_{i\in\zz,j\in\nn}$ such that
$$t^2Le^{-t^2L}(f)=\sum_{i\in\zz,j\in\nn}\lij\bij$$
almost everywhere on $\rnn$,
$\lij=2^i\|\chi_\Bij\|_\vp$ and
\begin{equation}\label{4-22x}
\sup_{i\in\zz}\ca(\{\lij\}_{j\in\nn},\{\Bij\}_{j\in\nn})
\ls \|t^2Le^{-t^2L}(f)\|_{W\!T^{p(\cdot)}(\rnn)}\sim \|f\|_{W\!H_L^{p(\cdot)}(\rn)}.
\end{equation}
Moreover, by the bounded holomorphic functional calculi for $L$, we know that
\begin{equation*}
f=C_{(\Phi,M)}\int_0^\fz(t^2L)^{M+1}\Phi(t\sqrt L)\lf(t^2Le^{-t^2L}(f)\r)\,\frac{dt}t
=\Pi_{\Phi,L}(t^2Le^{-t^2L}(f))
\end{equation*}
in $L^2(\rn)$, where $C_{(\Phi,M)}$ is the positive constant same as
in \eqref{4-26x}. Therefore, we have
\begin{equation}\label{4-22y}
f=\Pi_{\Phi,L}(t^2Le^{-t^2L}(f))=\sum_{i\in\zz,j\in\nn}\lij\Pi_{\Phi,L}(\bij)
=:\sum_{i\in\zz,j\in\nn}\lij\aij,
\end{equation}
where the above equalities hold true in $L^2(\rn)$. By Lemma \ref{l-3-18a}, we know that
$\aij$ is a $(p(\cdot),2,M)_L$-atom associated with
$2\Bij$ up to a positive harmless constant multiple. Thus,
\eqref{4-22y} is a weak atomic $(p(\cdot),2,M)_L$-representation of $f$,
which, combined with \eqref{4-22x}, implies that
$f$ belongs to $\mathbb{W\!H}_{L,{\rm at},M}^{p(\cdot)}(\rn)$.
Moreover, by Remarks \ref{r-1-30}(i) and \ref{r-1-30x}, we conclude that
\begin{align*}
\|f\|_{W\!H_{L,{\rm at},M}^{p(\cdot)}(\rn)}
&\ls \sup_{i\in\zz}\lf\|\lf\{\sum_{j\in\nn}
\lf[\frac{\lij\chi_{2\Bij}}{\|\chi_{2\Bij}\|_\vp}\r]^{p_-}\r\}^{\frac1{p_-}}\r\|_\vp\\
&\ls \sup_{i\in\zz}\lf\|\lf\{\sum_{j\in\nn}
\chi_{2\Bij}\r\}^{\frac1{p_-}}\r\|_\vp\ls
 \sup_{i\in\zz}\lf\|\lf\{\sum_{j\in\nn}
\chi_{\Bij}\r\}^{\frac1{p_-}}\r\|_\vp\\
&\sim \sup_{i\in\zz}\ca(\{\lij\}_{j\in\nn},\{\Bij\}_{j\in\nn})
\ls \|f\|_{W\!H_L^{p(\cdot)}(\rn)}.
\end{align*}
This finishes the proof of Theorem \ref{t-1-29}.
\end{proof}

\section{Non-tangential maximal function characterizations\label{s-6}}

The main purpose of this section is to establish non-tangential maximal function characterizations
of the weak Hardy space $W\!H_L^{p(\cdot)}(\rn)$ with $L$ as in Remark \ref{r-3-30}(ii),
namely, $L$ is a non-negative self-adjoint operator on $L^2(\rn)$ whose heat kernels
satisfying the Gaussian upper bound estimates.
To this end, we begin with the following notion.

\begin{defn}\label{d-3-16}
\begin{enumerate}
\item[(i)] Let $\phi\in\cs(\rr)$ be an even function with $\phi(0)=1$ and
$L$ be as in Remark \ref{r-3-30}(ii). For any $a\in(0,\fz)$ and $f\in L^2(\rn)$, the
\emph{non-tangential maximal function $\phi_{L,\bigtriangledown,a}(f)$} of $f$
is defined by setting, for any $x\in\rn$,
$$\phi_{L,\bigtriangledown,a}(f)(x):=\sup_{t\in(0,\fz),|y-x|<at}
|\phi(t\sqrt L)(f)(y)|.$$
A function $f\in L^2(\rn)$ is said to be in the set
$\mathbb{W\!H}_{L,\max}^{p(\cdot),\phi,a}(\rn)$
if $\phi_{L,\bigtriangledown,a}^\ast(f)\in L^{p(\cdot)}(\rn)$; moreover, define
$\|f\|_{W\!H_{L,\max}^{p(\cdot),\phi,a}(\rn)}
:=\|\phi_{L,\bigtriangledown,a}^\ast(f)\|_{W\!L^{p(\cdot)}(\rn)}$.
Then the \emph{variable weak Hardy space}
$W\!H_{L,\max}^{p(\cdot),\phi,a}(\rn)$ is defined to be the completion of
$\mathbb{W\!H}_{L,\max}^{p(\cdot),\phi,a}(\rn)$ with respect to the quasi-norm
$\|\cdot\|_{W\!H_{L,\max}^{p(\cdot),\phi,a}(\rn)}$.

Particularly, if $\phi(x):=e^{-x^2}$ for any $x\in\rr$, then we use $f_{L,\btd}^\ast$ to denote
$\phi_{L,\btd,1}^\ast(f)$ and, in this case, denote the space $W\!H_{L,\max}^{p(\cdot),\phi,1}(\rn)$
simply by $W\!H_{L,\max}^{p(\cdot)}(\rn)$.

\item[(ii)] For any $f\in L^2(\rn)$, define the \emph{grand non-tangential maximal function
$\mathcal G_{L,\btd}^\ast(f)$} of $f$ by setting, for any $x\in \rn$,
$$\mathcal G_{L,\btd}^\ast(f)(x):=\sup_{\phi\in \mathcal F(\rr)}\phi_{L,\btd,1}^\ast(f)(x),$$
where $\mathcal F(\rr)$ denotes the set of all even functions $\phi\in\cs(\rr)$ satisfying
$\phi(0)\neq 0$ and
$$\sum_{k=0}^N\int_\rr(1+|x|)^N\lf|\frac{d^k\phi(x)}{dx^k}\r|^2\,dx\le 1$$
with $N$ being a large enough number depending on $p(\cdot)$ and $n$.
Then the \emph{variable weak Hardy space $W\!H_{L,\max}^{p(\cdot),\cf}(\rn)$} is defined in the same
way as $W\!H_{L,\max}^{p(\cdot),\phi,a}(\rn)$ but with $\phi_{L,\btd,a}^\ast$ replaced by
$\mathcal G_{L,\btd}^\ast(f)$.
\end{enumerate}
 \end{defn}

The non-tangential maximal function characterizations of the space $W\!H_L^{p(\cdot)}(\rn)$
are stated as follows.

\begin{thm}\label{t-3-16}
Let $p(\cdot)\in C^{\log}(\rn)$ with $p_+\in(0,1]$, $M\in(\frac n2[\frac1{p_-}-\frac12],\fz)$
and $L$ be a linear operator on $L^2(\rn)$ as in Remark \ref{r-3-30}(ii),
where $p_-$ and $p_+$ are given by \eqref{2.1x}.
Then, for any $a\in(0,\fz)$ and $\phi$ as in Definition \ref{d-3-16}, the spaces
$W\!H_{L,{\rm at},M}^{p(\cdot)}(\rn)$, $W\!H_{L,\max}^{p(\cdot),\cf}(\rn)$
and $W\!H_{L,\max}^{p(\cdot),\phi,a}(\rn)$ coincide with equivalent quasi-norms.
\end{thm}

By Theorems \ref{t-3-16} and \ref{t-1-29}, we immediate obtain the following
corollary, the details being omitted.

\begin{cor}\label{c-4-25}
Under the same notation as in Theorem \ref{t-3-16}, it holds true that
the spaces
$$W\!H_{L}^{p(\cdot)}(\rn),\ W\!H_{L,\max}^{p(\cdot),\cf}(\rn)
\ {\rm and}\ W\!H_{L,\max}^{p(\cdot),\phi,a}(\rn)$$
coincide with equivalent quasi-norms.
\end{cor}

To prove Theorem \ref{t-3-16}, we need several lemmas. The following conclusion
is just \cite[Corollary 3.3]{yyyz16}.

\begin{lem}\label{l-3-16}
Let $p(\cdot)\in C^{\log}(\rn)$ satisfy $1<p_-\le p_+<\fz$. Then the Hardy-Littlewood maximal
operator $\cm$ is bounded on $\wlv$.
\end{lem}

\begin{lem}\label{l-3-16x}
Let $L$ be a linear operator on $L^2(\rn)$ as in Remark \ref{r-3-30}(ii),
$p(\cdot)\in C^{\log}(\rn)$, $\az_1,\ \az_2\in(0,\fz)$
and $\vz\in \cs(\rr)$ be an even function with $\vz(0)=1$. If $\lz\in(n/p_-,\fz)$, then there
exists a positive constant $C$ such that, for any $f\in L^2(\rn)$,
\begin{equation}\label{3-16x}
\|\vz_{L,\btd,\az_1}^\ast(f)\|_{\wlv}
\le C\lf(1+\frac{\az_1}{\az_2}\r)^\lz\|\vz_{L,\btd,\az_2}^\ast(f)\|_{\wlv}.
\end{equation}
\end{lem}
\begin{proof}
For any Borel measurable function $F$ on $\rr_+^{n+1}$, define its \emph{non-tangential maximal function
$M_{\az,\btd}(F)$} with aperture $\az\in(0,\fz)$, by setting, for any $x\in\rn$,
$$M_{\az,\btd}(F)(x):=\sup_{t\in(0,\fz),|y-x|<\az t}|F(y,t)|.$$
By the proof of \cite[Lemma 3.1]{zy15}, we know that, for any $\lz\in(n/p_-,\fz)$,
\begin{equation}\label{3-16y}
\sup_{\gfz{t\in(0,\fz)}{y\in\rn}}|F(y,t)|\lf(1+\frac{|x-y|}{\az t}\r)^{-\lz}
\ls \lf(1+\frac{\az_1}{\az_2}\r)^\lz\lf\{\cm\lf([M_{\az_2,\btd}(F)]^{n/\lz}\r)(x)
\r\}^{\frac \lz n}.
\end{equation}
From this, the boundedness of $\cm$ on $\wlv$ (see Lemma \ref{l-3-16}) and the fact that,
for any $x\in\rn$,
$$M_{\az_1,\btd}(F)(x)\ls \sup_{\gfz{t\in(0,\fz)}{y\in\rn}}|F(y,t)|
\lf(1+\frac{|x-y|}{\az t}\r)^{-\lz},$$
we deduce that
\begin{equation*}
\|M_{\az_1,\btd}(F)\|_{\wlv}\ls \lf(1+\frac{\az_1}{\az_2}\r)^{\lz}\|M_{\az_2,\btd}(F)\|.
\end{equation*}
By this and taking $F(x,t):=\vz(t\sqrt L)(f)(x)$ with any $x\in\rn$ and $t\in(0,\fz)$,
we conclude that \eqref{3-16x} holds true, which completes the proof
of Lemma \ref{l-3-16x}.
\end{proof}

\begin{lem}\label{l-3-16y}
Let $L$ be a linear operator on $L^2(\rn)$ as in Remark \ref{r-3-30}(ii),
$p(\cdot)\in C^{\log}(\rn)$, $\psi_1$, $\psi_2\in \cs(\rr)$
be even functions with $\psi_1(0)=1=\psi_2(0)$ and $\az_1,\ \az_2\in(0,\fz)$. Then there
exists a positive constant $C\in(0,\fz)$, depending on $\psi_1,\ \psi_2,\ \az_1$ and $\az_2$,
such that, for any $f\in L^2(\rn)$,
\begin{equation*}
\|(\psi_1)_{L,\btd,\az_1}^\ast (f)\|_{\wlv}\le C \|(\psi_2)_{L,\btd,\az_2}^\ast (f)\|_{\wlv}.
\end{equation*}
\end{lem}
\begin{proof}
Let $\psi:=\psi_1-\psi_2$. Then it is easy to see that
$$\|(\psi_1)_{L,\btd,\az_1}^\ast (f)\|_{\wlv}
\ls \|\psi_{L,\btd,\az_1}^\ast (f)\|_{\wlv}
+\|(\psi_2)_{L,\btd,\az_1}^\ast (f)\|_{\wlv}.$$
Thus, to prove this lemma, it suffices to show that
\begin{equation}\label{3-16z}
\|\psi_{L,\btd,\az_1}^\ast (f)\|_{\wlv}\ls\|(\psi_2)_{L,\btd,\az_2}^\ast (f)\|_{\wlv}
\end{equation}
due to Lemma \ref{l-3-16x}. Moreover, by Lemma \ref{l-3-16x} again, we may assume that
$\az_1=1=\az_2$. By \cite[(3.3) and (3.4)]{sy16}, we find that, for any $\lz\in(n/p_-,2M)$
and $x\in\rn$,
$$\psi_{L,\btd,1}^\ast(f)(x)\ls \sup_{t\in(0,\fz),y\in\rn}|\psi_2(t\sqrt L)(y)|
\lf(1+\frac{|x-y|}{t}\r)^{-\lz},$$
which, combined with \eqref{3-16y} and Lemma \ref{l-3-16}, implies that \eqref{3-16z} holds true.
This finishes the proof of Lemma \ref{l-3-16y}.
\end{proof}

We now prove Theorem \ref{t-3-16}.

\begin{proof}[Proof of Theorem \ref{t-3-16}]
We first show the following inclusion
\begin{equation}\label{3-18v}
\lf[\wha\cap L^2(\rn)\r]\subset \lf[\whf\cap L^2(\rn)\r].
\end{equation}
Let $f\in \wha\cap L^2(\rn)$. Then $f$ has a weak atomic $(p(\cdot),2,M)_L$-representation
\begin{equation}\label{3-18z}
f=\sum_{i\in\zz}\sum_{j\in\nn}\lz_{i,j}\aij\quad {\rm in}\quad L^2(\rn),
\end{equation}
where $\{\aij\}_{i\in\zz,j\in\nn}$ are $(p(\cdot),2,M)_L$-atoms associated, respectively, with balls
$\{\Bij\}_{i\in\zz,j\in\nn}$ and
$$\lij:=2^i\|\chi_{\Bij}\|_\lv$$
satisfying that, for any $i\in\zz$, $\sum_{j\in\nn}\chi_{c\Bij}\ls 1$ with $c\in(0,1]$ being
a positive constant independent of $i$ and $j$, and
\begin{equation*}
\sup_{i\in\zz}\ca\lf(\{\lij\}_{j\in\nn},\{\Bij\}_{j\in\nn}\r)\ls \|f\|_{\wha}.
\end{equation*}

To prove $f\in \whf$, it suffices to show that
\begin{equation}\label{3-18y}
\|f\|_{\whf}:=\|\mathcal G_{L,\btd}^\ast(f)\|_{\wlv}\ls \|f\|_{\wha}.
\end{equation}

For any $\phi \in\cs(\rr)$ and $x\in\rr$, let $\wz\psi(x):=[\phi(0)]^{-1}\phi(x)-e^{-x^2}$.
Then it is easy to see that
\begin{equation}\label{3-18x}
\mathcal G_{L,\btd}^\ast(f)\le \sup_{\phi\in\cf}\wz\psi_{L,\btd,1}^\ast(f)+f_{L,\btd}^\ast.
\end{equation}
Moreover, by some arguments similar to those used in the proofs of \cite[(3.3) and (3.4)]{sy16},
we find that, for any $\lz_0\in(n/p_-, 2M)$ and $x\in\rn$,
\begin{align*}
\sup_{\phi\in\cf}\wz\psi_{L,\btd,1}^\ast(f)(x)
&\ls \sup_{s\in(0,\fz),z\in\rn}e^{-s^2L}(f)(z)\lf(1+\frac{|x-z|}{s}\r)^{-\lz_0}\\
&\sim \sup_{\gfz{s\in(0,\fz)}{z\in\rn}}|e^{-s^2L}(f)(z)|\lf(1+\frac{|x-z|}{s}\r)^{-\lz_0}
\lf\{\chi_{|z-x|<s}(z)+\sum_{k=1}^\fz\chi_{2^{k-1}s\le |x-z|<2^k}(z)\r\}\\
&\ls\sum_{k=0}^\fz2^{-k\lz_0}\sup_{s\in(0,\fz),\ |z-x|<2^ks}|e^{-s^2L}(f)(z)|.
\end{align*}
From this, together with the Fatou lemma in $\wlv$ (see \cite[Lemma 2.12]{yyyz16}) and
Remark \ref{r-3-18}(ii), we deduce that
\begin{align*}
\lf\|\sup_{\phi\in\cf}\wz\psi_{L,\btd,1}^\ast(f)\r\|_{\wlv}
&\ls \lf\|\lim_{N\to\fz}\sum_{k=0}^N2^{-k\lz_0}\sup_{s\in(0,\fz),\ |z-\cdot|<2^ks}
|e^{-s^2L}(f)(z)\r\|_{\wlv}\\
&\ls \liminf_{N\to\fz}\lf\|\sum_{k=0}^N2^{-k\lz_0}\sup_{s\in(0,\fz),\ |z-\cdot|<2^ks}
|e^{-s^2L}(f)(z)\r\|_{\wlv}\\
&\ls \liminf_{N\to\fz}\lf\{\sum_{k=0}^N2^{-k\lz_0 v}
\lf\|\sup_{s\in(0,\fz),\ |z-\cdot|<2^ks}
|e^{-s^2L}(f)(z)\r\|_{\wlv}^v\r\}^{1/v},
\end{align*}
which, together with Lemma \ref{l-3-16x} and choosing $\lz\in(n/p_-,\lz_0)$, implies that
\begin{align*}
\lf\|\sup_{\phi\in\cf}\wz\psi_{L,\btd,1}^\ast(f)\r\|_{\wlv}
&\ls \lf\{\sum_{k=0}^\fz 2^{-k\lz_0 v}2^{k\lz v}
\lf\|\sup_{s\in(0,\fz),\ |z-\cdot|<s}
|e^{-s^2L}(f)(z)\r\|_{\wlv}^v\r\}^{1/v}\\
&\ls \|f_{L,\btd}^\ast\|_{\wlv}.
\end{align*}
Thus, by this, \eqref{3-18x} and Remark \ref{r-3-18}(i), we know that
\begin{equation*}
\|\mathcal G_{L,\btd}^\ast\|_{\wlv}\le
\lf\|\sup_{\phi\in\cf}\wz\psi_{L,\btd,1}^\ast(f)+f_{L,\btd}^\ast\r\|_{\wlv}
\ls \|f_{L,\btd}^\ast\|_{\wlv}.
\end{equation*}
Therefore, to prove \eqref{3-18y}, we only need to show
\begin{equation}\label{3-18u}
\|f_{L,\btd}^\ast\|_{\wlv}\ls \|f\|_{\wha}.
\end{equation}
Since $f$ has a weak atomic $(p(\cdot),2,M)_L$-representation \eqref{3-18z}, it follows that
\begin{align*}
\|f_{L,\btd}^\ast\|_{\wlv}\le
\lf\|\sum_{i\in\zz}\sum_{j\in\nn}|\lij|(\aij)_{L,\btd}^\ast\r\|_{\wlv}.
\end{align*}
Moreover, by an argument similar to that used in the proof of \cite[(3.10)]{zy15}, we
find that, for any $x\in (4\Bij)^\com$,
\begin{align*}
(\aij)_{L,\btd}^\ast(x)\ls \frac{[r_\Bij]^{n+\delta}}{|x-x_{\Bij}|^{n+\delta}}
\frac1{\|\chi_\Bij\|_{\lv}},
\end{align*}
where $\delta\in(n[\frac1{p_-}-1],2M)$, $r_{\Bij}$ and $x_{\Bij}$ denote, respectively,
the radius and the center of the ball $\Bij$. Then, by an argument similar to that used in the proof of
\eqref{3-18w}, we conclude that \eqref{3-18u} holds true. This finishes the proof of
\eqref{3-18v}.

Next, we prove
\begin{equation}\label{3-18b}
\lf[W\!H_{L,\max}^{p(\cdot),\phi,a}(\rn)\cap L^2(\rn)\r]\subset\lf[\wha\cap L^2(\rn)\r].
\end{equation}
Let $f\in W\!H_{L,\max}^{p(\cdot),\phi,a}(\rn)\cap L^2(\rn)$. To prove $f\in \wha$,
by Lemma \ref{l-3-16y}, we only need to show that
\begin{equation}\label{3-18a}
\|f\|_{\wha}\ls \|f\|_{W\!H_{L,\max}^{p(\cdot)}(\rn)}.
\end{equation}
Let $\phi_0\in\cs(\rr)$, $\supp (\phi_0)\subset (-1,1)$, $\Phi:=\widehat{\phi_0}$. Then,
by the functional calculi, we know that
$$f=C_{(M,\Phi)}\int_0^\fz (t^2L)^{M+1}\Phi(t\sqrt L)e^{-t^2L}(f)\,\frac{dt}t
\quad {\rm in}\quad L^2(\rn),$$
where $C_{(M,\Phi)}$ is a positive constant same as in \eqref{4-26x}.
Following \cite[p.\,476]{sy16}, define a function $\eta$ by setting,
$$\eta(x):=C_{(M,\Phi)}\int_1^\fz (tx)^{2M+2}\Phi(tx)e^{-t^2x^2}\,\frac{dt}t,\quad \forall
\,x\in\rr\backslash\{0\},$$
and $\eta(0):=1$. Then $\eta\in\cs(\rr)$ is an even function and, for any $a,\ b\in\rr$,
$$C_{(M,\Phi)}\int_a^b (t^2L)^{M+1}\Phi(t\sqrt L)e^{-t^2L}(f)\,\frac{dt}t=
\eta(a\sqrt L)(f)-\eta(b\sqrt L)(f).$$
For any $x\in\rn$, let
$$\mathcal N_L^\ast(f)(x):=\sup_{t\in(0,\fz),\ |y-x|<5\sqrt{nt}}
\lf[|t^2Le^{-t^2L}(f)(y)|+|\eta(t\sqrt L)(f)(y)|\r]$$
and, for any $i\in\zz$,
$O_i:=\{x\in\rn:\ \mathcal N_L^\ast(f)(x)>2^i\}$. Then $f$ has the following decomposition
$$f=\sum_{i\in\zz}\sum_{j\in\nn}\lij\aij\quad{\rm in}\quad L^2(\rn)$$
(see \cite[pp.\,476-479]{sy16} for more details). Here, for any $i\in\zz$ and $j\in\nn$,
$\aij$ is a $(p(\cdot),2,M)_L$-atom, associated with the ball $30B_{i,j}$, and
$\lij:=2^i\|30B_{i,j}\|_{\lv}$, where $\{B_{i,j}\}_{j\in\nn}$ is the Whitney
decomposition of $O_i$ satisfying that $O_i=\bigcup_{j\in\nn}B_{i,j}$, $\{B_{i,j}\}_{j\in\nn}$
have disjoint interiors and the property of finite overlap.
Moreover, by Remark \ref{r-1-30}(i) and Lemma \ref{l-3-16y},
we conclude that
\begin{align*}
\sup_{i\in\zz}\lf(\{\lij\}_{i\in\zz,j\in\nn},\{30B_{i,j}\}_{i\in\zz,j\in\nn}\r)
&\sim \sup_{i\in\zz}2^i\lf\|\lf\{\sum_{j\in\nn}\chi_{30B_{i,j}}\r\}^\frac1{p_-}\r\|_\lv
\ls \sup_{i\in\zz}2^i\lf\|\sum_{j\in\nn}\chi_{B_{i,j}}\r\|_\lv\\
&\ls \sum_{i\in\zz}2^i\|\chi_{O_i}\|_{\lv}\ls \|\mathcal N_L^\ast(f)\|_{\wlv}
\ls \|f\|_{W\!H_{L,\max}^{p(\cdot)}(\rn)},
\end{align*}
which implies \eqref{3-18a}. This finishes the proof of \eqref{3-18b}.

Finally, the inclusion
$$\lf[\whf\cap L^2(\rn)\r]\subset \lf[W\!H_{L,\max}^{p(\cdot),\phi,a}(\rn)\cap L^2(\rn)\r]$$
is just a consequence of Lemma \ref{l-3-16y}. From this, together with \eqref{3-18v} and
\eqref{3-18b}, we deduce that
$$\lf[\wha\cap L^2(\rn)\r]= \lf[W\!H_{L,\max}^{p(\cdot),\phi,a}(\rn)\cap L^2(\rn)\r]
=\lf[\whf\cap L^2(\rn)\r],$$
which, combined with a density argument, implies that
the spaces $$\wha,\ W\!H_{L,\max}^{p(\cdot),\phi,a}(\rn)\ {\rm and }\ \whf$$
coincide with
equivalent quasi-norms. This finishes the proof of Theorem \ref{t-3-16}.
\end{proof}

\section{Boundedness of the Riesz transform $\nabla L^{-1/2}$\label{s-7}}

In this section, let $L$ be the second-order divergence form elliptic operator as in
Remark \ref{r-3-30}(i). Define the Riesz transform $\nabla L^{-1/2}$ by setting,
for any $f\in L^2(\rn)$ and $x\in\rn$,
$$\nabla L^{-1/2}(f)(x):=\frac1{2\sqrt \pi}\int_0^\fz \nabla e^{-s L}(f)(x)\,\frac{ds}{\sqrt s}.$$
Then we prove that the operator $\nabla L^{-1/2}$ is bounded from $W\!H_L^{p(\cdot)}(\rn)$
to the variable weak Hardy space $W\!H^{p(\cdot)}(\rn)$ introduced in \cite{yyyz16}.
Recall that the domain of $L^{1/2}$ coincides with the Sobolev
space $H^1(\rn)$ (see \cite[Theorem 1.4]{ahlmt02}).
Therefore, for any $f\in L^2(\rn)$, $L^{-1/2}(f)\in H^1(\rn)$ and
$\nabla L^{-1/2}(f)$ stands for the distributional derivatives of $L^{-1/2}(f)$.

For any $N\in\nn$, define
\begin{align*}
\cf_N(\rn):=\lf\{\psi\in\cs(\rn):\ \sum_{\bz\in\zz_+^n,\,|\bz|\le N}\sup_{x\in\rn}(1+|x|)^N
\lf|D^\bz\psi(x)\r|\le 1\r\},
\end{align*}
where, for any $\bz:=(\bz_1,\,\ldots,\,\bz_n)\in\zz_+^n$,
$|\bz|:=\bz_1+\cdots +\bz_n$ and $D^\bz:=(\frac{\partial}{\partial x_1})^{\bz_1}\cdots
(\frac{\partial}{\partial x_n})^{\bz_n}$.
For any $N\in\nn$ and $f\in\cs'(\rn)$, the \emph{grand maximal function $f_{N,+}^\ast$} of $f$
is defined by setting, for any $x\in\rn$,
\begin{align*}
f_{N,+}^\ast(x):=\sup\{|\psi_t\ast f(x)|:\ t\in(0,\,\fz),\, \psi\in \cf_N(\rn)\},
\end{align*}
where, for any $t\in(0,\,\fz)$ and $\xi\in\rn$, $\psi_t(\xi):=t^{-n}\psi(\xi/t)$.

\begin{defn}\label{def hardy}
Let $p(\cdot)\in C^{\log}(\rn)$ and $N\in(\frac{n}{\min\{p_-,1\}}+n+1,\,\fz)$ with $p_-$ as in
\eqref{2.1x}.
Then the \emph{variable weak Hardy space} $W\!H^{p(\cdot)}(\rn)$ is defined by setting
\begin{align*}
W\!H^{p(\cdot)}(\rn):=\lf\{f\in\cs'(\rn):\ \|f\|_{W\!H^{p(\cdot)}(\rn)}:=
\|f_{N,+}^\ast\|_{\wlv}<\fz\r\}.
\end{align*}
\end{defn}

\begin{rem}
It was proved in \cite[Theorem 3.7]{yyyz16} that the space $\whv$ is independent of the choice
of $N\in(\frac{n}{\min\{p_-,1\}}+n+1,\,\fz)$.
\end{rem}

Moreover, the space $\whv$ admits the following molecular characterization, which
was obtained in \cite[Theorem 5.3]{yyyz16}.

Let $p(\cdot)\in C^{\log}(\rn)$, $q\in (1,\fz]$, $s\in{\zz}_+$ and
$\uc\in (0, \fz)$. A measurable function $m$ is called a
$(p(\cdot), q, s,\epsilon)$-\emph{molecule} associated
with some ball $B\subset\rn$ if
\begin{enumerate}
\item[{\rm (i)}] for any $j\in\zz_+$, $\|m\|_{L^q(U_j(B))}\le 2^{-j\uc}
|U_j(B)|^{\frac{1}{q}}\|\chi_B\|_{\lv}^{-1}$,
where $U_j(B)$ is given by \eqref{eq ujb}.

\item[{\rm (ii)}] $\int_{\rn}m(x)x^\beta dx=0$ for all
$\beta\in\zz_+^n$ with $|\beta|\leq s$.
\end{enumerate}

\begin{defn}\label{mod2}
Let $p(\cdot)\in C^{\log}(\rn)$, $q\in(1,\fz]$, $\uc\in (0, \fz)$
and $s\in(\frac{n}{p_-}-n-1,\fz)\cap{\zz}_+$ with $p_-$ as in \eqref{2.1x}.
The \emph{variable weak molecular Hardy space}
$W\!H_{\rm mol}^{p(\cdot),q,s,\uc}(\rn)$ is defined as the space of all $f\in\cs'(\rn)$
which can be decomposed as  $f=\sum_{i\in\zz}\sum_{j\in\nn}\lij\mij$
in $\cs'(\rn)$, where $\{\mij\}_{i\in\zz,j\in\nn}$ is a sequence of
$(p(\cdot),q,s,\uc)$-molecules associated, respectively, with balls
$\{\Bij\}_{i\in\zz,j\in\nn}$,
$$\{\lij\}_{i\in\zz,j\in\nn}:=\{\wz A2^i\|\chi_{\Bij}\|_{\lv}\}_{i\in\zz,j\in\nn}$$ with
$\wz A$ being a positive constant independent of $i$ and $j$, and there exist
positive constants $A$ and $C$ such that, for any $i\in\zz$ and
$x\in\rn, \sum_{j\in\nn}\chi_{C\Bij}(x)\le A$.

Moreover, for any $f\in W\!H_{\rm mol}^{p(\cdot),q,s,\uc}(\rn)$, define
$$\|f\|_{W\!H_{\rm mol}^{p(\cdot),q,s,\uc}(\rn)}:=\inf\lf[\sup_{i\in\zz}{\lf\|\lf\{\sum_{j\in\nn}
\lf[\frac{\lij\chi_\Bij}
{\|\chi_\Bij\|_{\lv}}\r]^{\underline{p}}\r\}^{1/{\underline{p}}}\r\|
_{\lv}}\r],\noz\\$$
where the infimum is taken over all decompositions of $f$ as above.
\end{defn}

\begin{lem}\label{mothm1}
Let $p(\cdot)\in C^{\log}(\rn)$, $q\in(\max\{p_+,1\},\fz]$,
$s\in(\frac{n}{p_-}-n-1,\fz)\cap{\zz}_+$ and
$\uc\in (n+s+1, \fz)$, where $p_+$ and $p_-$ are as in \eqref{2.1x}.
Then $\whv=W\!H_{\rm mol}^{p(\cdot),q,s,\uc}(\rn)$ with equivalent quasi-norms.
\end{lem}

\begin{rem}\label{r-4-2}
If the variable exponent $p(\cdot)$ satisfies that $\frac n{n+1}< p_-\le p_+\le 1$,
then $\frac n{p_-}-n-1<0$. Thus, in this case, we can take $s=0$ and $\uc>n+1>\frac n{p_-}$
in Lemma \ref{mothm1}.
\end{rem}

\begin{thm}\label{t-4-2}
Let $p(\cdot)\in C^{\log}(\rn)$ with $\frac n{n+1}<p_-\le p_+\le1$ and $L$ be the
second-order divergence form elliptic operator
as in \eqref{eq op}, where $p_-$ and $p_+$ are given by \eqref{2.1x}.
Then the Riesz transform $\nabla L^{-1/2}$ is bounded from
$W\!H_L^{p(\cdot)}(\rn)$ to $\whv$.
\end{thm}

\begin{proof}
Let $f\in W\!H_L^{p(\cdot)}(\rn)\cap L^2(\rn)$. Then, by Proposition \ref{p-3-28}, we find that,
for any given $\uc\in(\frac n{p_-},\fz)$ and $M\in \nn\cap (\frac n2[\frac1{p_-}-\frac12],\fz)$,
there exist sequences $\{\mij\}_{i\in\zz,j\in\nn}$
 of $(p(\cdot),M,\uc)_L$-molecules associated, respectively, with balls
$\{\Bij\}_{i\in\zz,j\in\nn}$ and numbers $\{\lij\}_{i\in\zz,j\in\nn}$ satisfying
\begin{enumerate}
\item[(i)] $f=\sum_{i\in\zz,j\in\nn} \lij m_{i,j}$ in $L^2(\rn)$;
\item[(ii)] for any $i\in\zz$ and $j\in\nn$, $\lij:=2^i\|\chi_{\Bij}\|_{\vp}$;
\item[(iii)]
\begin{equation}\label{4-2x}
\sup_{i\in\zz}\ca(\{\lij\}_{j\in\nn},\{\Bij\}_{j\in\nn})\ls \|f\|_{W\!H_L^{p(\cdot)}(\rn)};
\end{equation}
\item[(iv)] there exists a positive constant $c\in(0,1]$ such that, for any $x\in\rn$ and $i\in\zz$,
$\sum_{j\in\nn}\chi_{c\chi_\Bij}(x)\le M_0$ with $M_0$ being a positive constant independent of $x$
and $i$.
\end{enumerate}

Moreover, by an argument similar to that used in the proof of \cite[Theorem 5.17]{yzz17},
we know that, for any $i\in\zz$ and $j\in\nn$,
\begin{equation*}
\|\nabla L^{-1/2}(\mij)\|_{L^2(U_k(\Bij))}
\le 2^{-k\min\{\uc,2M+\frac n2\}}|2^k\Bij|^\frac12\|\chi_{\Bij}\|_\lv^{-1},
\quad \forall\,k\in\zz_+,
\end{equation*}
and $\int_\rn \nabla L^{-1/2}(\mij)(x)\,dx=0$, namely, $\nabla L^{-1/2}(\mij)$
is a $(p(\cdot),2,0,\uc)$-molecule associated with $\Bij$. From this, together with
the boundedness of $\nabla L^{-1/2}$ on $L^2(\rn)$ and \eqref{4-2x}, we deduce that
\begin{align*}
\nabla L^{-1/2}(f)=\sum_{i\in\zz}\sum_{j\in\nn}\lij \nabla L^{-1/2}(\mij)
\end{align*}
belongs to $W\!H_{\rm mol}^{p(\cdot),2,0,\uc}(\rn)$. By this, \eqref{4-2x},
Lemma \ref{mothm1} and Remark
\ref{r-4-2}, we conclude that
\begin{align*}
\|\nabla L^{-1/2}(f)\|_{\whv}
&\sim \|\nabla L^{-1/2}(f)\|_{W\!H_{\rm mol}^{p(\cdot),2,0,\uc}(\rn)}\\
&\ls \sup_{i\in\zz}\ca(\{\lij\}_{j\in\nn},\{\Bij\}_{j\in\nn})\ls \|f\|_{W\!H_L^{p(\cdot)}(\rn)}.
\end{align*}
Therefore, the Riesz transform $\nabla L^{-1/2}$ is bounded from
$W\!H_L^{p(\cdot)}(\rn)$ to $\whv$. This finishes the proof of Theorem \ref{t-4-2}.
\end{proof}

\medskip

\noindent Ciqiang Zhuo

\medskip

\noindent Key Laboratory of High Performance Computing and Stochastic
Information Processing (HPCSIP) (Ministry of Education of China),
College of Mathematics and Statistics, Hunan Normal University,
Changsha, Hunan 410081, People's Republic of China

\smallskip

\noindent {\it E-mail}: \texttt{cqzhuo87@hunnu.edu.cn} (C. Zhuo)

\bigskip

\noindent Dachun Yang (Corresponding author)
\medskip

\noindent  Laboratory of Mathematics and Complex Systems
(Ministry of Education of China),
School of Mathematical Sciences, Beijing Normal University,
Beijing 100875, People's Republic of China

\smallskip

\noindent {\it E-mail}: \texttt{dcyang@bnu.edu.cn} (D. Yang)


\begin{thebibliography}{10}

\bibitem{am02} E. Acerbi and G. Mingione, Regularity results for stationary
electro-rheological fluids, Arch. Ration. Mech. Anal. 164 (2002), 213-259.

\vspace{-0.3cm}

\bibitem{am05} E. Acerbi and G. Mingione,
Gradient estimates for the $p(x)$-Laplacean system,
J. Reine Angew. Math. 584 (2005), 117-148.

\vspace{-0.3cm}

\bibitem{adm96}
 D. Albrecht, X. T. Duong and A. McIntosh,
Operator theory and harmonic analysis, in: Instructional Workshop on Analysis and Geometry,
Part III (Canberra, 1995), 77-136,
Proc. Centre Math. Appl. Austral. Nat. Univ., 34, Austral. Nat. Univ., Canberra, 1996.

\vspace{-0.3cm}

\bibitem{abdr17}
V. Almeida, J. J. Betancor, E. Dalmasso and L. Rodr\'iguez-Mesa,
Local Hardy spaces with variable exponents associated to non-negative self-adjoint operators
satisfying Gaussian estimates, arXiv: 1712.06710.

\vspace{-0.3cm}

\bibitem{Ao42} T. Aoki,
Locally bounded linear topological spaces,
Proc. Imp. Acad. Tokyo 18 (1942), 588-594.

\vspace{-0.3cm}

\bibitem{adm05}
P. Auscher, X. T. Duong and A. McIntosh,
Boundedness of Banach space valued singular integral operators and Hardy spaces,
unpublished manuscript, 2005.

\vspace{-0.3cm}

\bibitem{ahlmt02}
P. Auscher, S. Hofmann, M. Lacey, A. McIntosh and Ph. Tchamitchian,
The solution of the Kato square root problem for second order elliptic operators on $\rn$,
Ann. of Math. (2) 156 (2002), 633-654.

\vspace{-0.3cm}

\bibitem{bckyy13a}
 T. A. Bui, J. Cao, L. D. Ky, D. Yang and S. Yang, Musielak-Orlicz-Hardy
spaces associated with operators satisfying reinforced off-diagonal
estimates, Anal. Geom. Metr. Spaces  1 (2013), 69-129.

\vspace{-0.3cm}

\bibitem{bckyy13}
T. A. Bui, J. Cao, L. D. Ky, D. Yang and S. Yang,
Weighted Hardy spaces associated with operators satisfying reinforced off-diagonal estimates,
Taiwanese J. Math. 17 (2013), 1127-1166.

\vspace{-0.3cm}

\bibitem{cchy15} J. Cao, D.-C. Chang, H. Wu and D. Yang, Weak Hardy spaces $W\!H^p_L(\rn)$
associated to operators satisfying k-Davies-Gaffney estimates, J. Nonlinear Convex Anal.
16 (2015), 1205-1255.

\vspace{-0.3cm}

\bibitem{cgzl08}
Y. Chen, W. Guo, Q. Zeng and Y. Liu,
A nonstandard smoothing in reconstruction of apparent diffusion coefficient profiles
from diffusion weighted images,
Inverse Probl. Imaging 2 (2008), 205-224.

\vspace{-0.3cm}

\bibitem{cms85} R. R. Coifman, Y. Meyer and E. M. Stein, Some new function
spaces and their applications to harmonic analysis,
J. Funct. Anal. 62 (1985), 304-335.

\vspace{-0.3cm}

\bibitem{cruz03} D. Cruz-Uribe, The Hardy-Littlewood maximal operator on variable-$L^p$
 spaces, in: Seminar of Mathematical Analysis (Malaga/Seville, 2002/2003), 147-156,
 Colecc. Abierta, 64, Univ. Sevilla Secr. Publ., Seville, 2003.

\vspace{-0.3cm}

\bibitem{cfbook} D. V. Cruz-Uribe and A. Fiorenza, Variable Lebesgue Spaces.
Foundations and Harmonic Analysis, Applied and Numerical Harmonic Analysis,
Birkh\"auser/Springer, Heidelberg, 2013.

\vspace{-0.3cm}

\bibitem{cf06} D. Cruz-Uribe, A. Fiorenza, J. M. Martell and C. P\'erez,
The boundedness of classical operators on variable $L^p$ spaces,
Ann. Acad. Sci. Fenn. Math. 31 (2006), 239-264.

\vspace{-0.3cm}

\bibitem{cw14} D. Cruz-Uribe and L.-A. D. Wang,
Variable Hardy spaces,
Indiana Univ. Math. J. 63 (2014), 447-493.

\vspace{-0.3cm}

\bibitem{din04} L. Diening, Maximal function on generalized Lebesgue
 spaces $\lv$, Math. Inequal. Appl. 7 (2004), 245-253.

\vspace{-0.3cm}

\bibitem{dhr11} L. Diening, P. Harjulehto, P. H\"ast\"o and
M. R{$\mathring{\rm u}$}\v{z}i\v{c}ka,
Lebesgue and Sobolev Spaces with
Variable Exponents, Lecture Notes in Math. 2017,
Springer, Heidelberg, 2011.

\vspace{-0.3cm}

\bibitem{dl13}
X. T. Duong and J. Li,
Hardy spaces associated to operators
satisfying Davies-Gaffney estimates and bounded holomorphic functional calculus,
J. Funct. Anal. 264 (2013), 1409-1437.

\vspace{-0.3cm}

\bibitem{dy05}
X. T. Duong and L. Yan,
New function spaces of BMO type, the John-Nirenberg inequality, interpolation, and applications,
Comm. Pure Appl. Math. 58 (2005), 1375-1420.

\vspace{-0.3cm}

\bibitem{dy051}
X. T. Duong and L. Yan,
Duality of Hardy and BMO spaces associated with operators with heat kernel bounds,
J. Amer. Math. Soc. 18 (2005), 943-973.

\vspace{-0.3cm}

\bibitem{fs86} R. Fefferman and F. Soria, The space weak $H^{1}$,
Studia Math. 85 (1986), 1-16.

\vspace{-0.3cm}

\bibitem{g09-1} L. Grafakos,
Classical Fourier Analysis, Third edition,
Graduate Texts in Math. 249, Springer, New York, 2014.

\vspace{-0.3cm}

\bibitem{gly08}
L. Grafakos, L. Liu and D. Yang,
Maximal function characterizations of Hardy spaces on RD-spaces and their
applications,
Sci. China Ser. A 51 (2008), 2253-2284.

\vspace{-0.3cm}

\bibitem{hhl08}
P. Harjulehto, P. H\"{a}st\"{o} and V. Latvala,
Minimizers of the variable exponent, non-uniformly convex Dirichlet energy,
J. Math. Pures Appl. (9) 89 (2008), 174-197.

\vspace{-0.3cm}

\bibitem{He14}
D. He, Square function characterization of weak Hardy spaces,
J. Fourier Anal. Appl. 20 (2014), 1083-1110.

\vspace{-0.3cm}

\bibitem{hlmmy} S. Hofmann, G. Lu, D. Mitrea, M. Mitrea and L. Yan,
Hardy spaces associated to non-negative self-adjoint operators satisfying Davies-Gaffney estimates,
Mem. Amer. Math. Soc. 214 (2011), no. 1007, vi+78 pp.

\vspace{-0.3cm}

\bibitem{hm09}
 S. Hofmann and S. Mayboroda,
Hardy and BMO spaces associated to divergence form elliptic operators,
Math. Ann. 344 (2009), 37-116.

\vspace{-0.3cm}

\bibitem{hmm11}
S. Hofmann, S. Mayboroda and A. McIntosh,
Second order elliptic operators with complex bounded measurable coefficients in $L^p$,
Sobolev and Hardy spaces,
Ann. Sci. \'Ec. Norm. Sup\'er. (4) 44 (2011), 723-800.

\vspace{-0.3cm}

\bibitem{hyy} S. Hou, D. Yang and S. Yang, Lusin area function and molecular
characterizations of Musielak-Orlicz Hardy spaces and their applications,
Commun. Contemp. Math. 15 (2013), 1350029, no. 6, 37 pp.

\vspace{-0.3cm}

\bibitem{Iz10}
M. Izuki, Vector-valued inequalities on Herz spaces and characterizations
of Herz-Sobolev spaces with variable exponent,
Glas. Mat. Ser. III 45 (65) (2010), 475-503.

\vspace{-0.3cm}

\bibitem{jy10}
 R. Jiang and D. Yang,
New Orlicz-Hardy spaces associated with divergence form elliptic operators,
J. Funct. Anal. 258 (2010), 1167-1224.

\vspace{-0.3cm}

\bibitem{jy11}
R. Jiang and D. Yang,
Orlicz-Hardy spaces associated with operators satisfying Davies-Gaffney estimates,
Commun. Contemp. Math. 13 (2011), 331-373.

\vspace{-0.3cm}

\bibitem{jyz09} R. Jiang, D. Yang and Y. Zhou, Orlicz-Hardy spaces associated
with operators, Sci. China Ser. A 52 (2009), 1042-1080.

\vspace{-0.3cm}

\bibitem{kr91} O. Kov\'a\v{c}ik and J. R\'akosn\'{\i}k, On spaces
$L^{p(x)}$ and $W^{k,p(x)}$, Czechoslovak Math. J. 41 (116) (1991), 592-618.

\vspace{-0.3cm}

\bibitem{m86}
A. McIntosh,
Operators which have an $H^\fz$ functional calculus, in:
Miniconference on Operator Theory and Partial Differential Equations (North Ryde, 1986),
210-231, Proc. Centre Math. Anal. Austral. Nat. Univ., 14, Austral. Nat. Univ., Canberra, 1986.

\vspace{-0.3cm}

\bibitem{ns12} E. Nakai and Y. Sawano, Hardy spaces with variable exponents
and generalized Campanato spaces, J. Funct. Anal. 262 (2012), 3665-3748.

\vspace{-0.3cm}

\bibitem{or31} W. Orlicz, \"Uber konjugierte Exponentenfolgen, Studia Math.
3 (1931), 200-211.

\vspace{-0.3cm}

\bibitem{ou05}
E. M. Ouhabaz,
Analysis of heat equations on domains,
London Mathematical Society Monographs Series, 31,
Princeton University Press, Princeton, NJ, 2005.

\vspace{-0.3cm}

\bibitem{Ro57} S. Rolewicz,
On a certain class of linear metric spaces,
Bull. Acad. Polon. Sci. Cl. Trois. 5 (1957), 471-473.

\vspace{-0.3cm}

\bibitem{rm00} M. R$\mathring {\rm u}$\v{z}i\v{c}ka,
 Electrorheological Fluids: Modeling and Mathematical Theory,
Lecture Notes in Math. 1748, Springer-Verlag, Berlin, 2000.

\vspace{-0.3cm}

\bibitem{su09}
M. Sanch\'{o}n and J. Urbano,
Entropy solutions for the $p(x)$-Laplace equation,
Trans. Amer. Math. Soc. 361 (2009), 6387-6405.

\vspace{-0.3cm}

\bibitem{Sa13} Y. Sawano, Atomic decompositions of Hardy spaces with variable exponents
and its application to bounded linear operators,
Integral Equations Operator Theory 77 (2013), 123-148.

\vspace{-0.3cm}

\bibitem{sy16}
L. Song and L. Yan,
A maximal function characterization for Hardy spaces associated to
nonnegative self-adjoint operators satisfying Gaussian estimates,
Adv. Math. 287 (2016), 463-484.

\vspace{-0.3cm}

\bibitem{yan08} L. Yan, Classes of Hardy spaces associated with operators,
duality theorem and applications, Trans. Amer. Math. Soc. 360 (2008), 4383-4408.

\vspace{-0.3cm}

\bibitem{yyyz16}
X. Yan, D. Yang, W. Yuan and C. Zhuo, Variable weak Hardy spaces and their applications,
J. Funct. Anal. 271 (2016), 2822-2887.

\vspace{-0.3cm}

\bibitem{yyz16}
D. Yang, W. Yuan and C. Zhuo, A survey on some variable function spaces,
in: Function Spaces and Inequalities, 299-335, Springer Proc. Math. Stat., 206, Springer, Singapore, 2017.

\vspace{-0.3cm}

\bibitem{yzz17} D. Yang, J. Zhang and C. Zhuo, Variable Hardy spaces associated
with operators satisfying Davies-Gaffney estimates, Proc. Edinb. Math. Soc. (to appear).

\vspace{-0.3cm}

\bibitem{yz16}
D. Yang and C. Zhuo,
Molecular characterizations and dualities of variable exponent Hardy spaces
associated with operators,
Ann. Acad. Sci. Fenn. Math. 41 (2016), 357-398.

\vspace{-0.3cm}

\bibitem{yzn16} D. Yang, C. Zhuo and E. Nakai,
Characterizations of variable exponent Hardy spaces via Riesz transforms,
Rev. Mat. Complut. 29 (2016), 245-270.

\vspace{-0.3cm}

\bibitem{zy15}
C. Zhuo and D. Yang,
Maximal function characterizations of
variable Hardy spaces associated with non-negative self-adjoint operators
satisfying Gaussian estimates,
Nonlinear Anal. 141 (2016), 16-42.

\vspace{-0.3cm}

\bibitem{zyl16} C. Zhuo, D. Yang and Y. Liang, Intrinsic square function characterizations of
Hardy spaces with variable exponents, Bull. Malays. Math. Sci. Soc. 39 (2016), 1541-1577.

\end{thebibliography}
\end{document}